\documentclass[11pt]{article}
\usepackage[top=3cm,bottom=2cm,left=2.1cm,right=2.1cm,marginparwidth=1.75cm]{geometry}
\usepackage[format=plain, font=footnotesize]{caption}
\usepackage{amsmath, amsthm, nccmath, amssymb, mathtools, hyperref}
\hypersetup{
	colorlinks=true,       % false: boxed links; true: colored links
	linkcolor=orange,        % color of internal links
	citecolor=purple,        % color of links to bibliography
	filecolor=magenta,     % color of file links
	urlcolor=blue
}
\usepackage[shortlabels]{enumitem}
\usepackage{graphicx} 
\usepackage[dvipsnames]{xcolor}
\usepackage{authblk}
\usepackage{xcolor}
\usepackage{listings}
\usepackage{url}
\usepackage{cleveref}
\usepackage{systeme}
\usepackage{array}
\usepackage{bigints}
\usepackage{comment}
\usepackage{lipsum}
\usepackage{subcaption}
\newcommand{\Cheb}{\mathrm{Cheb}}
\newcommand\norm[1]{\left\lVert#1\right\rVert}
\usepackage{tikz-cd}
\tikzset{
    circ/.style={circle, fill=black, inner sep=2pt, node contents={}}
}

\newtheorem{theorem}{Theorem}
\numberwithin{theorem}{section}

\newtheorem{lemma}[theorem]{Lemma}
\newtheorem{corollary}[theorem]{Corollary}

\theoremstyle{definition}
\newtheorem{definition}[theorem]{Definition}
\newtheorem{remark}[theorem]{Remark}

\numberwithin{equation}{section}
\newcommand{\diam}{\mathrm{diam}}
\crefname{equation}{}{}
\crefname{equation}{}{}
\usepackage{tikz}
\crefname{figure}{Figure}{Figure}
\crefname{section}{Section}{Section}
\crefname{lemma}{Lemma}{Lemma}
\crefname{proposition}{Proposition}{Proposition}
\crefname{theorem}{Theorem}{Theorem}
\crefname{corollary}{Corollary}{Corollarie}
\crefname{definition}{Definition}{Definition}
\crefname{notation}{Notations}{Notation}
\crefname{remark}{Remark}{Remark}
\crefname{claim}{Claim}{Claim}
\crefname{assumption}{Assumption}{Assumption}

\newcommand{\R}{\mathbb{R}}

\newcommand{\li}{\left}
\newcommand{\re}{\right}

\definecolor{OurRed}{rgb}{1, 0, 0}

\usepackage{todonotes}

\begin{document}

\title{A note on the rate of convergence of integration schemes for closed surfaces}

\date{}
\author[1]{Gentian Zavalani}
\author[2]{Elima Shehu}
\author[3]{Michael Hecht}

\affil[1,3]{{\small Center for Advanced Systems Understanding (CASUS), D-02826 Görlitz, Germany}}
\affil[1,3]{{\small Helmholtz-Zentrum Dresden-Rossendorf, D-01328 Dresden, Germany.}}
\affil[1]{{\small Technische Universit\"{a}t Dresden}}

\affil[2]{{\small Max Planck Institute for Mathematics in the Sciences, Leipzig, Germany}}
\affil[2]{{\small Osnabr\"uck University, Osnabr\"uck, Germany}}

\maketitle

\begin{abstract}
In this paper, we issue an error analysis for integration over discrete surfaces using the surface parametrization presented in \cite{CurvedGrid}  as well as prove why even-degree polynomials utilized for approximating both the smooth surface and the integrand exhibit a higher convergence rate than odd-degree polynomials. Additionally, we provide some numerical examples that illustrate our findings and propose a potential approach that overcomes the problems associated with the original one.
\end{abstract}


% \MH{what's with author emails ...corresponding authors marks }
% \GZ{No clue! However, once the paper is accepted, probably they will ask to write things using their standard template. }
\section{Introduction}
 Many applications, including mathematical physics and mathematical biology \cite{YANG09}, require accurate approximation of integrals on curved surfaces. The concept of surface integration is a fundamental procedure in a wide range of numerical methods, including the boundary integral method, the finite element method, the surface finite element method, and the finite volume method. As a result of piecewise linear approximations of the surface and integrand, integration over discrete surfaces (such as surface triangulation) is typically only of first- or second-order accuracy~\cite{Georg98}. To improve the order of accuracy, using the approach presented in \cite{CurvedGrid} a polynomial approximation of the geometry of a smooth closed embedded hypersurface $\mathcal{M}$ is considered, where a local interpolation polynomial of the closest point projection map $\pi$ for each element $T\in \mathcal{T}^{1}_{h}$ will be constructed, with $\mathcal{T}^{1}_{h}$ denoting a conforming triangulation of $\mathcal{M}$. 
 % The mapping $\pi$ is a bijective mapping from the piecewise flat surface to the smooth surface. 

The contribution of this paper is twofold. First, we deliver a parametrization of smooth, closed
surfaces, enabling to numerically approximate surface integrals based on discrete triangulations. That is approximately providing a decomposition $\mathcal{M}=\bigcup_{i=1}^{n}V_{i}$ into non-overlapping regions $V_i$, given as images of regular maps, see Fig.~\eqref{fig.app_frame}
 \begin{equation}\label{sis1}
    \psi_i : \sigma  \longrightarrow V_i \subseteq \mathcal{M}\,, 
 \end{equation}
 where $\sigma$ is a reference simplex defined as $\sigma:=\lbrace\left(s,t\right),0\leq s\leq 1, 0\leq t\leq 1-s\rbrace\,$ in~$\mathbb{R}^2$. 
 % 
 % we provide a smooth parametrization of the surface
 % \MH{again Vi surface}
 % $V_i$ in $\mathbb{R}^3$ such that $\partial_s\psi_{i}(s,t)\times \partial_t\psi_{i}(s,t)\neq 0$ at all points, 
 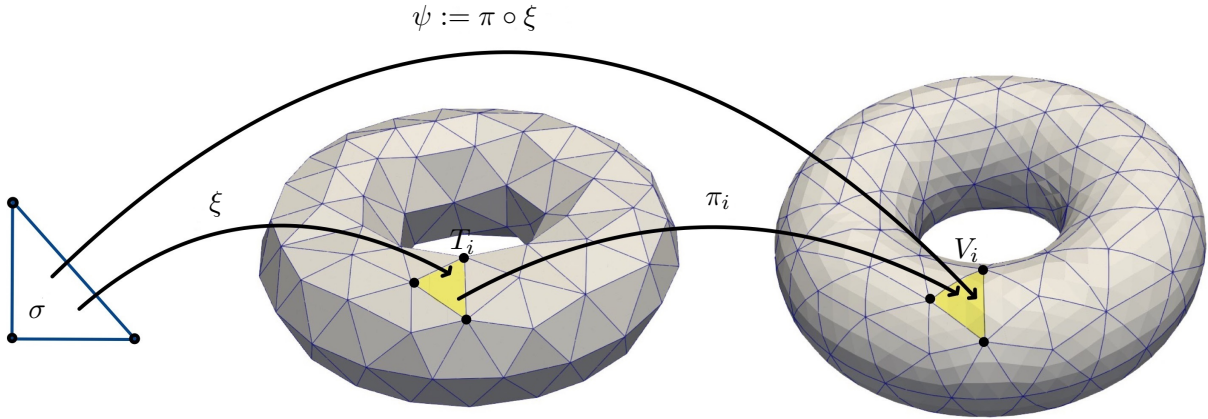
\begin{figure}[t!]
    \centering
    \begin{tikzpicture}
        % Include the image
        \node[inner sep=0pt] at (0,0) {\includegraphics[clip,width=1.0\columnwidth]{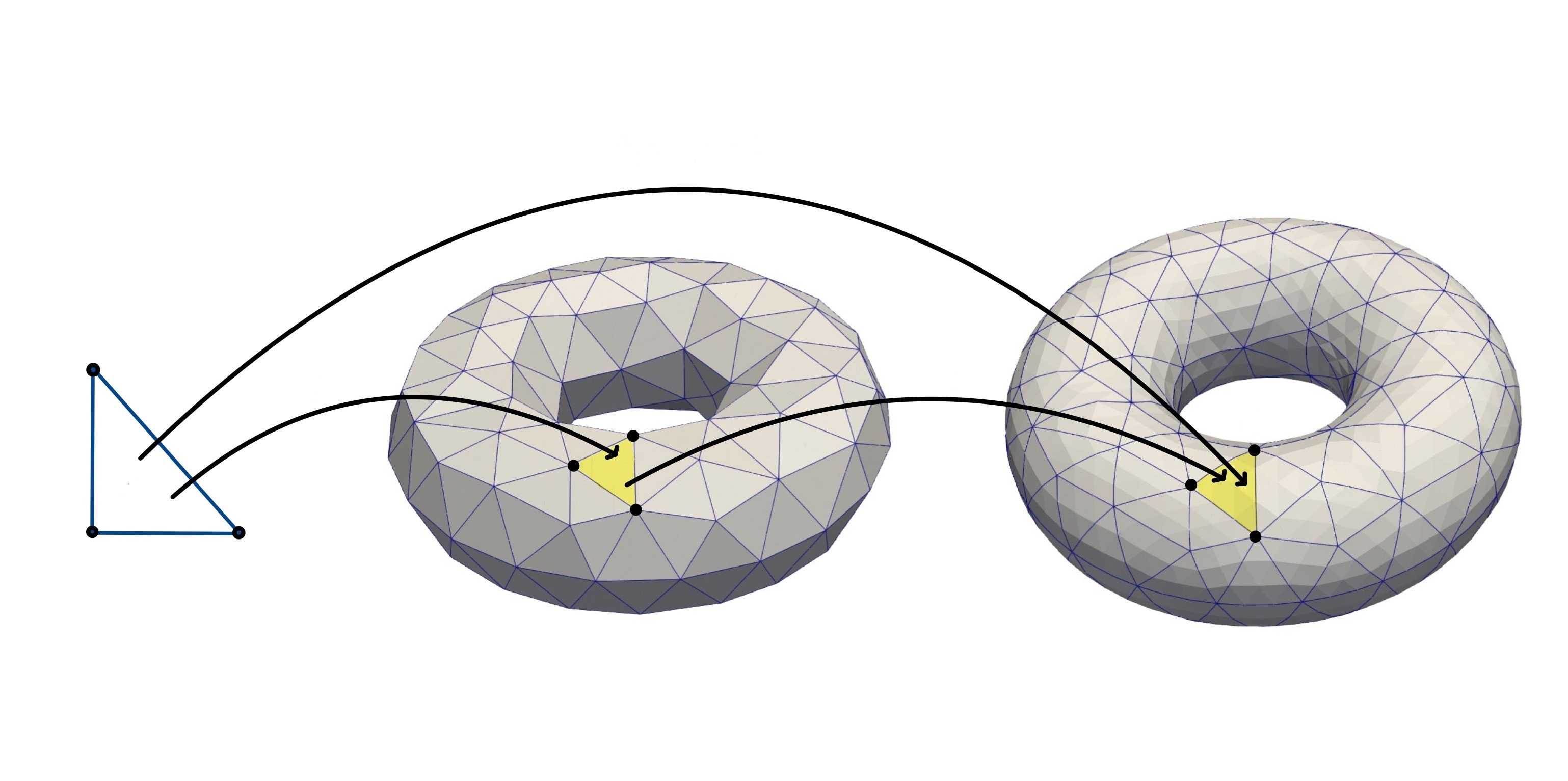}};
        
        % Add text and symbols
    
        \node[anchor=north west] at (1.4,0.5) {$\pi_i$};

        \node[anchor=north west] at (4.7,-0.1) {$V_i\,$};

        \node[anchor=north west] at (-2.5,3.0) {$\psi:=\pi\circ\xi$};
        \node[anchor=north west] at (-2.,0.01) {$T_i$};
        \node[anchor=north west] at (-5.2,0.52) {$\xi$};
        \node[anchor=north west] at (-7.6,-1) {$\sigma$};
    \end{tikzpicture}
  \vspace{-20pt} % Adjust the value to decrease/increase the gap
  \caption{Representation of a smooth surface parametrization, where every region $V_{i}$ forms a curved triangle.}
  \label{fig.app_frame}
\end{figure} 

Consequently, the surface integration problem becomes:
\begin{equation}
        \int_{\mathcal{M}}fdS =\sum_{i=1}^{n}\int_{V_i}fdS=\sum_{i=1}^n\int_{0}^1\int_{0}^{1-s} f\big(\psi_i(s,t)\big)g_{i}(s,t)dsdt\,,
\end{equation}
where $g_{i}(s,t) = \sqrt{\det(J^{T}J)} $ with $J = [\partial_{s}\psi_{i}(s,t) , \partial_{t}\psi_{i}(s,t) ]$ and 
$\partial_{s}\psi_{i}(s,t),\,\partial_{t}\psi_{i}(s,t)$ denoting the partial derivatives of $\psi_i$ with respect to $s$ and $t$. We propose to approximate the function $f$ and each parametrization
 $\psi_i$ by polynomials of degree $k \in \mathbb{N}$. Consequently, the integral is approximated by numerically computing
 \begin{equation}\label{sis.1}
  \sum_{i=1}^n\int_{0}^1\int_{0}^{1-s} Q_{f,k}\big(\psi_i(s,t)\big)  \Tilde{g}_{i}(s,t)dsdt\,,
\end{equation}
where $\Tilde{g}_{i}(s,t)=\sqrt{\det(\Tilde{J}^{T}\Tilde{J})}$ with $J = [\partial_{s}Q_{\psi_{i},k}(s,t) , \partial_{t}Q_{\psi_{i},k}(s,t) ],$ and $Q_{\psi_{i},k}$ denoting a $k-$th order polynomial approximating the map $\psi_i$, whereas $Q_{f,k}$ is a $k-$th order polynomial approximating the integrand $f$.

The second part of our contribution is motivated by recent results of Ray et al. \cite{Ray2012},  where a stabilized least squares approximation, a blending procedure based on linear shape functions, and high-degree quadrature rules are combined into a method for integration over discrete surfaces. This method has an accuracy order of $\mathcal{O}\left(h^k+h^{5}\right)$ with $k$ denoting the polynomial degree and $h>0$ the mesh size 
(\cite{Ray2012}, Theorem 1). Based on their numerical experiments, they observe that even-degree polynomials exhibit a higher convergence rate than predicted by their theory. The aim of this paper is to publicize this phenomenon and expand our understanding of it with the aid of new numerical experiments (section 4) and a new Theorem \ref{main.thm}, which also provides a detailed analysis of errors for integration over discrete triangulated surfaces. 
% Among the methods that can benefit from such an error analysis are the finite element method, the surface finite element method, the boundary integral method, and the finite volume methods. 
 
 The article is structured as follows. Section 2 gives a mathematical description of the surface parametrization introduced in \cite{CurvedGrid}. In Section 3, we analyze the error using this approach and justify why even-degree polynomials exhibit higher convergence rates. In section 4 we conclude with some numerical examples to illustrate our findings and propose an alternative approach that overcomes some of the disadvantages of the original method.
Codes for reproducing the results of this manuscript are summarized and available in the repository \url{https://github.com/zavala92/code_paper}.

% \MH{This should become a institutional repo ...via CASUS}
\section{Polynomial Approximation for Closed Surfaces}\label{sec.1}

% \MH{We may discuss this first paragraph .... its a bit complicated to read.}
Let $\mathcal{M}$ be a smooth connected, orientable closed hypersurface, with smooth hypersurface we mean a $\mathrm{C}^{\infty}$ topological manifold that is second-countable, Hausdorff, and locally Euclidean of dimension 2, which is embedded in an ambient space of dimension 3. According to the Jordan–Brouwer decomposition theorem \cite{doi:10.1080/00029890.1988.11971963},  $\mathcal{M}$ divides $\mathbb{R}^{3}$ into an interior
and an exterior domain and we denote by $d$ the signed distance function to $\mathcal{M}$ oriented in such
a way that $d>0$ in the exterior, $d<0$ in the interior of $\mathcal{M}$. Let us denote the outward unit normal of $\mathcal{M}$ with $n(x)=\nabla d(x)$, where $ \nabla$ is the standard gradient in $\mathbb R^{3}$ (see \cite{Demlow09} for more details). 
Given $\mathcal{M}\subset \mathbb{R}^3$ for each point $x \in \mathcal{M}$ the tangent space $T_{x}\mathcal{M}\subseteq \mathbb{R}^3$ is a linear subspace. 
 The disjoint union of tangent spaces to $\mathcal{M}$ is the tangent bundle of $\mathcal{M}$:
\begin{equation*}
    T\mathcal{M}:=\amalg_{x\in \mathcal{M}}T_{x}\mathcal{M}=\{\left(x,y\right)|\, x\in\mathcal{M}\,\text{and}\, y\in T_{x}\mathcal{M}\}.
\end{equation*}
The normal space of $\mathcal{M}$ in $\mathbb{R}^3$ at $x$ is the orthogonal complement of the tangent space $T_{x}\mathcal{M}$, namely $N_{x}\mathcal{M}=\left(T_{x}\mathbb{R}^{3}\right)^{\perp}.$ The normal bundle is a smooth embedded submanifold of  $\mathbb{R}^3\times\mathbb{R}^3$ of dimension $3$, 
\begin{equation*}
    N\mathcal{M}:=\amalg_{x\in \mathcal{M}}N_{x}\mathcal{M}=\{\left(x,z\right)|\, x\in\mathcal{M}\,\text{and}\, z\in N_{x}\mathcal{M}\}.
\end{equation*}
Let $\delta:\mathcal{M}\longrightarrow \mathbb{R}_{+}$ be a positive, continuous function. Consider the following open tubular neighborhood of the normal bundle:
\begin{equation*}
    N_{\delta}=\{\left(x,y_{x}\right) \mid\, \norm{y_{x}}< \delta\left(x\right)\},
\end{equation*}
where $y_{x}$ is normal vector attached at $x$. The map $\mathcal{F}:N\mathcal{M}\longrightarrow \mathbb{R}^{3}$
$\left(x, y\right)\mapsto x+y$ is smooth and there
exists a $\delta$ such that the restriction $\mathcal{F}|_{N_{\delta}}$ becomes a diffeomorphism onto its image \cite{lee2013smooth}. Consequently, $\mathcal{N}_{\delta} = \mathcal{F}\left(N_{\delta}\right)$ is a $3-$dimensional, open, smooth,
embedded submanifold of $\mathbb{R}^{3}$  that forms a tubular neighborhood of $\mathcal{M}$.  

We will assume that we have a polyhedral surface $\mathcal{M}_h$ 
in Euclidean three-space, which is defined to be a compact subset $\mathcal{M}_h\subseteq \mathbb{R}^{3}$ homeomorphic to $\mathcal{M}$. This discrete surface is composed of finitely many triangles whose vertices are located in $\mathcal{M}$,  ensuring that each edge is contained in a certain (affine) line and each face is contained in a certain (affine) plane:
\begin{equation*}
    \mathcal{M}_h=\bigcup_{i=1}^{n}T_{i},\quad \mathcal{T}^{1}_{h}=\bigcup_{i=1}^{n}\{T_{i}\},
\end{equation*}
where each triangle $T_i$ is parameterized over a reference and  $\mathcal{T}^{1}_{h}$ is a collection of flat triangles with mesh size $h=\max_{T_{i}\in\mathcal{T}^{1}_{h}} \diam\big(T_{i}
\big)$. The collection $\mathcal{T}^{1}_{h}$ is called a conforming triangulation if for any $T_i, T_j\in \mathcal{T}^{1}_{h} $ with $T_i\neq T_j$ the intersection $T_i\cap T_j$ is either empty or a proper k-sub-simplex of $T_i\, \left(k<2\right)$. Let assume that $\mathcal{M}_h$ is contained in the tubular neighborhood $\mathcal{N}_{\delta}$. Under these conditions, we define a unique nonlinear  closest point projection map:
\begin{equation}\label{eq:closmp}
    \pi :\mathcal{N}_{\delta}	\supseteq\mathcal{M}_{h}\longrightarrow \mathcal{M}\subset \mathbb R^{3}
\end{equation}
of the form
\begin{equation*}
    \pi\left(x\right)=x-d\left(x\right)n\left(x\right)\,,
\end{equation*}
which assigns to every $x\in \mathcal{M}_{h}$ the closest point on $\mathcal{M}$, so that $\left(x-\pi(x)\right)\perp T_{\pi(x)}\mathcal{M},\,\, \forall x\in \mathcal{M}_{h}$. We assume that $\mathcal{N}_\delta$ has the additional property that for each point $x\in \mathcal{M}_{h}\subseteq \mathcal{N}_\delta$ there is a unique point $\pi(x)$ on $\mathcal{M}$ that minimizes the distance from $\mathcal{M}$ to $x$. 
% Such tubular neighborhoods always exist by suitably restricting of $\delta$. Since $\mathcal{M}$ is smooth, it is possible to choose $\delta> 0$ such that
% \begin{equation*}
    % \mathcal{N}_{\delta}=\{x\in \mathbb R^{3}\, \mid\, dist\left(x,\mathcal{M}\right)< \delta\},
% \end{equation*}
% where $dist\left(x,\mathcal{M}\right)$ is the Euclidean distance between $x$ and $\mathcal{M},$ and $\delta$ bounded by the reciprocal of the maximum over $\mathcal{M}$ of the moduli of all principal curvatures.
In other words, the computation of the closest point projection \eqref{eq:closmp} is a local minimizer problem \cite{Breiding_2021} in the sense: 
\begin{equation}
\pi:\mathcal{M}_{h}\ni x\mapsto  \mathop{argmin}_{y\in \mathcal{M}}\norm{ y-x},
\end{equation}
where the nonlinear projection  maps a
point $x\in \mathcal{M}_{h}$ to the point on $\mathcal{M}$ that minimizes the distance~to~$x$.

The accuracy of standard surface integration methods is limited to only first or second order due to the use of piecewise linear approximations of the surface geometry and the integrand. 
To obtain high order accuracy, we must construct a high order approximation, both, to the geometry of the surface and to the integrand. 

By relying on \cite{CurvedGrid} we give now the construction of a
surface $\mathcal{M}^{k}_{h}$,  which is locally parametrised over the reference simplex $\sigma$ by polynomials of degree $k$, interpolating the smooth surface $\mathcal{M}$. As aforementioned, we first consider a piecewise flat triangulation $\mathcal{M}_h$ of the smooth surface with vertices lying on~$\mathcal{M}$. To simplify the notation, the subscripts from the transformation map \eqref{sis1} and the triangulation of $\mathcal{M}_{h}$ are dropped. For any triangle $T\in\mathcal{T}^{1}_{h}$  with vertices $q_{1},q_{2},q_{3}\in\mathbb{R}^{3}$ we define the following map
\begin{equation*}
\xi:\sigma \longrightarrow T, \quad
    q_{i}=\xi\left(\hat{p_{i}}\right),\quad 1\leq i\leq 3 \quad \xi\left(s,t\right)= q_{1} +\big(q_{3}-q_{1}\big)s+\big(q_{2}-q_{1}\big)t
\end{equation*}
to be the affine linear parametrization which maps each vertex $\hat{p_{i}}$ of $\sigma$ to the vertex $q_{i}$ of $T.$  We denote the images of the non-vertex  nodes of $\sigma$ on the simplex $T$ by $\Bar{q_{i}}=\xi\left(\hat{p_{i}}\right),\,3< i\leq N\left(2,k\right)$, where $N\left(2,k\right)$ is the dimension of the vector space $\mathcal{P}_{2,k}\left(\sigma\right)$ of bivariate polynomials of degree $k$. Let~$\mathcal{L}^k_1\left(\hat{p}\right),\,\mathcal{L}^k_2\left(\hat{p}\right),\ldots,\mathcal{L}^k_{N\left(2,k\right)}\left(\hat{p}\right)$ be the local Lagrange basis functions of degree $k$ on $\sigma$ corresponding to the nodal points $\hat{p}_{1},\ldots \hat{p}_{N\left(2,k\right)}$. Set
\begin{equation*}
    p_{i}:=\pi\left(\Bar{q_{i}}\right)=\big(\pi\circ \xi\big)\left(\hat{p_{i}}\right)=\psi\left(\hat{p_{i}}\right),\quad \text{where}\quad \psi:=\pi\circ\xi,\quad 3< i\leq N\left(2,n\right)
\end{equation*}
and  define $Q_{\psi,k}$ to be a $k-$th order polynomial interpolation of the mapping $\psi$:
\begin{equation}\label{interp.}
    \mathcal{I}_{k}:C^{0}\big(\sigma,\mathbb{R}\big)\longrightarrow P_{k}\left(\sigma\right),\quad     \psi\mapsto Q_{\psi,k},\quad Q_{\psi,k}\left(\hat{p}\right):=\sum_{i=1}^{N\left(2,n\right)} p_{i}\mathcal{L}_{i}^{k}\left(\hat{p}\right)
\end{equation}
\begin{figure}[h]
\centering
    \begin{tikzpicture}
        % Include the image
        \node[inner sep=0pt] at (0,0) {\includegraphics[clip,width=1.0\columnwidth]{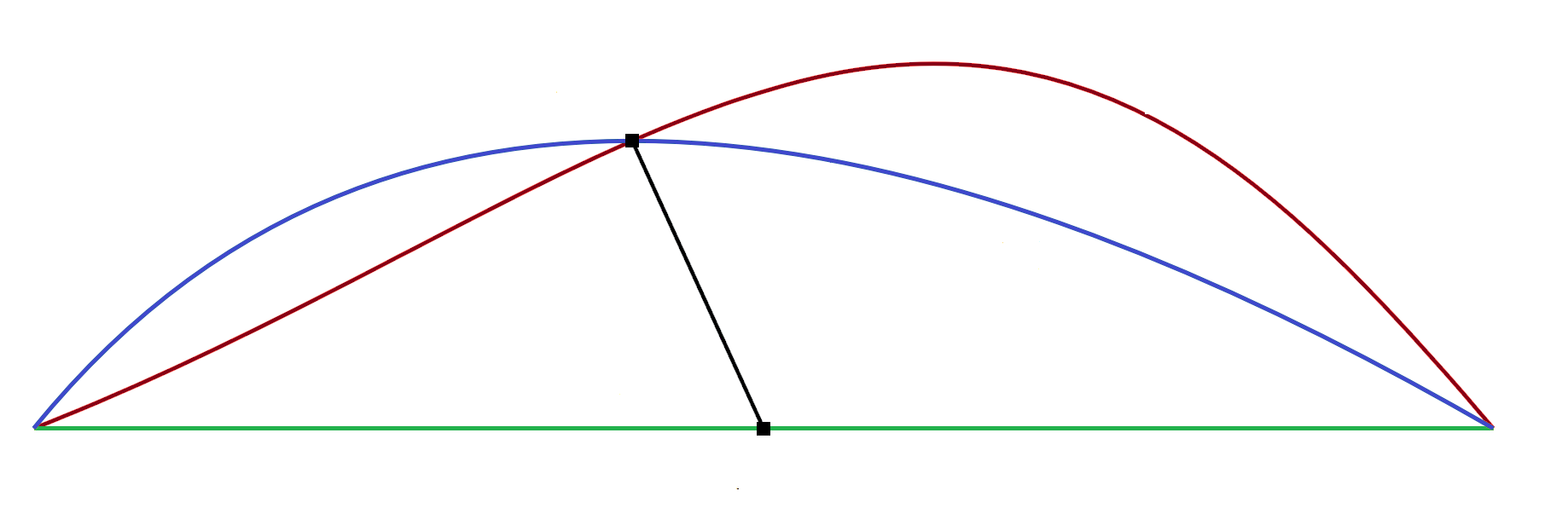}};
        
        % Add text and symbols
    
        \node[anchor=north west] at (1.4,0.5) {$\mathcal{M}^{2}_{h}$};

        \node[anchor=north west] at (2.7,2.5) {$\mathcal{M}$};

        \node[anchor=north west] at (-3,2.2) {$p_i=\pi\left(\bar{q_{i}}\right)$};
        \node[anchor=north west] at (-0.3,-2) {$\bar{q_{i}}$};
        \node[anchor=north west] at (-2.3,-1.2) {$\mathcal{M}_{h}$};
    \end{tikzpicture}
\caption{ Construction of the second order approximation of the smooth surface $\mathcal{M}^{2}_{h}$ (blue line). A simplex of a ‘base’ triangulation $\mathcal{M}_{h}$ (green line) is shown. The interpolation nodes, here the center $\bar{q_{i}}$ of an edge, are projected (grey line) onto the smooth surface $\mathcal{M}$ (red line) via the projection $\pi$. The projected nodal points $\pi\left(\bar{q_{i}}\right)$ and the vertices of $\mathcal{M}_{h}$ are then interpolated, giving the second order approximation of the smooth surface $\mathcal{M}^{2}_{h}$.}
\label{fig.projection}
\end{figure}
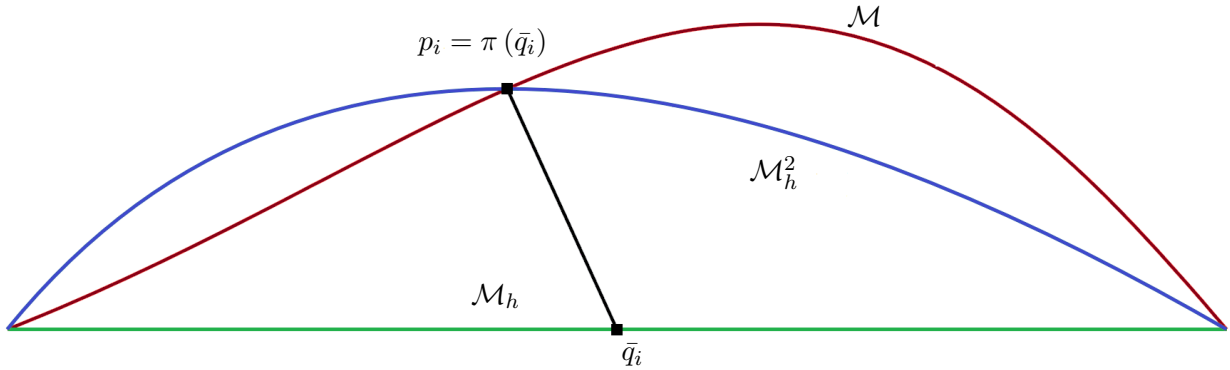\\
Now, 
\begin{equation*}
    Q_{\psi,k}:\sigma \rightarrow \Tilde{T}^{k}:= Q_{\psi,k}\big(\sigma\big)
\end{equation*}
defines a polynomial mapping 
 $Q_{\psi,k}$ by interpolating the points $p_i\in \mathcal{M}$ with the
Lagrange-polynomials $\mathcal{L}^k_1\left(\hat{p}\right),\,\mathcal{L}^k_2\left(\hat{p}\right),\ldots,\mathcal{L}^k_{N\left(2,k\right)}\left(\hat{p}\right).$
Thus, for every simplex $T\in\mathcal{T}^{1}_{h}$ we compute the projection $\pi(\bar{q}_{i})$ and define an isoparametric simplex $\Tilde{T}^{k}$ by applying Lagrange interpolation of order $k$ to the coordinates of the projected equidistant nodes (see Fig~\eqref{fig.projection}). 
% i.e., $\Tilde{T}^{k}:=\mathcal{I}_{k}\left(T\right)\,$ interpolates the surface $\mathcal{M}$ in its equidistant nodes, where $\Tilde{T}^{k}$ denote the  $k-$th order approximation $T$.
Furthermore,  if the base-triangulation $\mathcal{T}^{1}_h$ is fine enough, then the map $Q_{\psi,k}$ is a diffeomorphism. By differentiating the interpolation polynomial of the map $\psi$, we obtain:
\begin{equation}\label{jac}
    \partial_{s}
    Q_{\psi,k}\left(\hat{p}\right):=\sum_{i=1}^{N\left(2,n\right)} p_{i}\partial_{s}\mathcal{L}_{i}^{k}\left(\hat{p}\right),\quad     \partial_{t}
    Q_{\psi,k}\left(\hat{p}\right):=\sum_{i=1}^{N\left(2,n\right)} p_{i}\partial_{t}\mathcal{L}_{i}^{k}\left(\hat{p}\right).
\end{equation}
The Jacobian of the transformation is calculated using the equation \eqref{jac}. Given that $Q_{\psi,k}$ is a ~diffeomorphism $\forall\, T\in \mathcal{T}^{1}_{h}$, then  by union of non-overlapping mapped elements:
\begin{equation}
    \mathcal{M}_{h}^{k}:=\bigcup_{T\in \mathcal{T}^{1}_{h} }Q_{\psi,k}\left(T\right)=\bigcup_{T\in \mathcal{T}^{1}_{h},\,\hat{p}\in \sigma}\sum_{i=1}^{N\left(2,n\right)} \pi\big( \xi\left(\hat{p_{i}}\right)\big)\mathcal{L}_{i}^{k}\left(\hat{p}\right).
\end{equation}
% \begin{equation}
%     \mathcal{M}_{h}^{k}:=\bigcup_{T\in \mathcal{T}^{1}_{h} }\mathcal{I}_{k}\left(T\right)=\bigcup_{T\in \mathcal{T}^{1}_{h},\,\hat{p}\in \sigma}\sum_{i=1}^{N\left(2,n\right)} \pi\big( \xi\left(\hat{p_{i}}\right)\big)\mathcal{L}_{i}^{k}\left(\hat{p}\right).
% \end{equation}
Therefore, we have successfully obtained a $k$-th order discrete approximation $\mathcal{M}_{h}^{k}$ of the continuous surface $\mathcal{M}$. Fig.~\eqref{fig:L.param} depicts an illustration of this procedure applied to a torus and a sphere. In the specific scenario where $k=1$, we denote the discrete approximation as $\mathcal{M}_{h}=\mathcal{M}_{h}^{1}$.
\begin{figure}[!t]
\begin{subfigure}{.5\textwidth}
  \centering
  
  \includegraphics[scale=0.45]{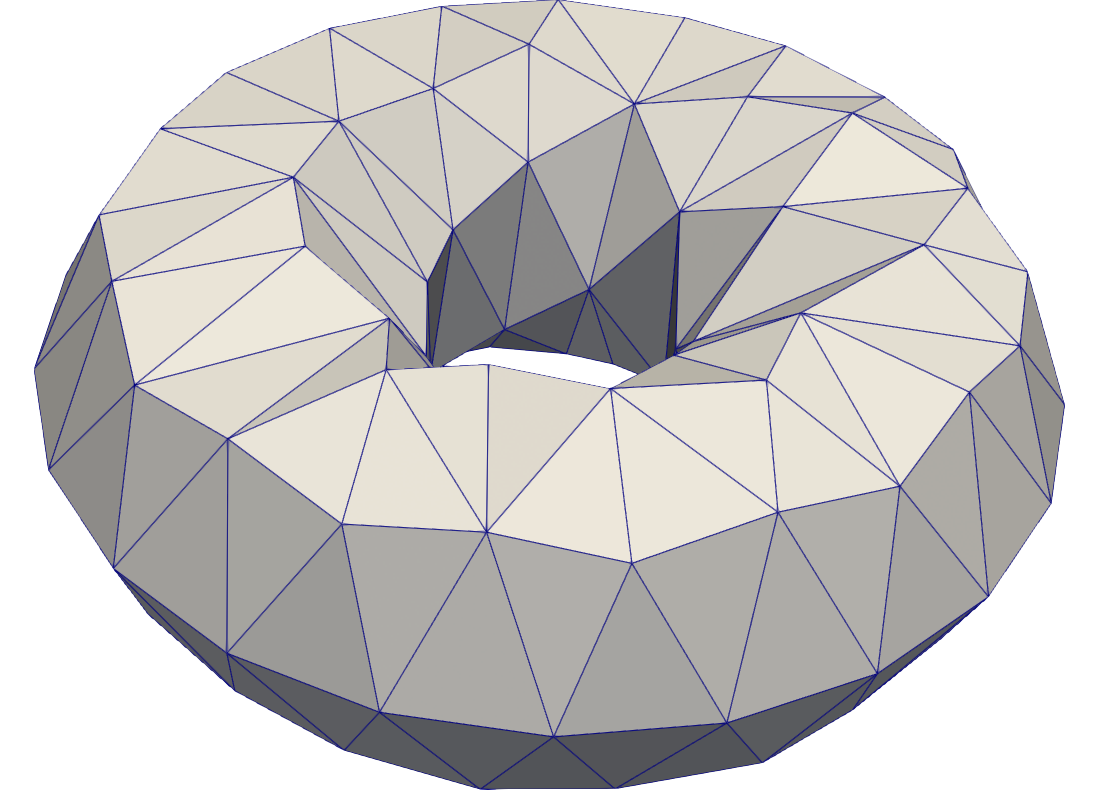}
  \hfill
  \caption{$\mathcal{M}_{h}$}
\end{subfigure}%
\begin{subfigure}{.5\textwidth}
  \centering
  
 \includegraphics[scale=0.45]{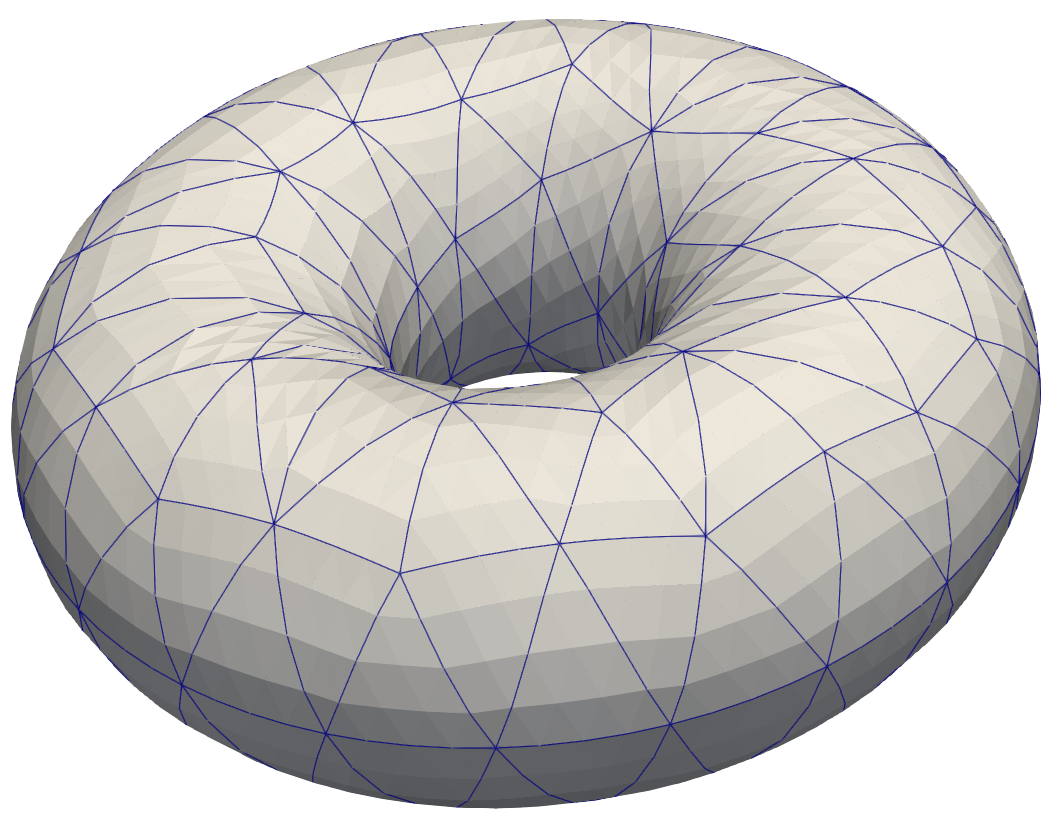}
 \hfill
  \caption{$\mathcal{M}_{h}^{4}$}
\end{subfigure}
\begin{subfigure}{.5\textwidth}
  \centering

  \includegraphics[scale=0.45]{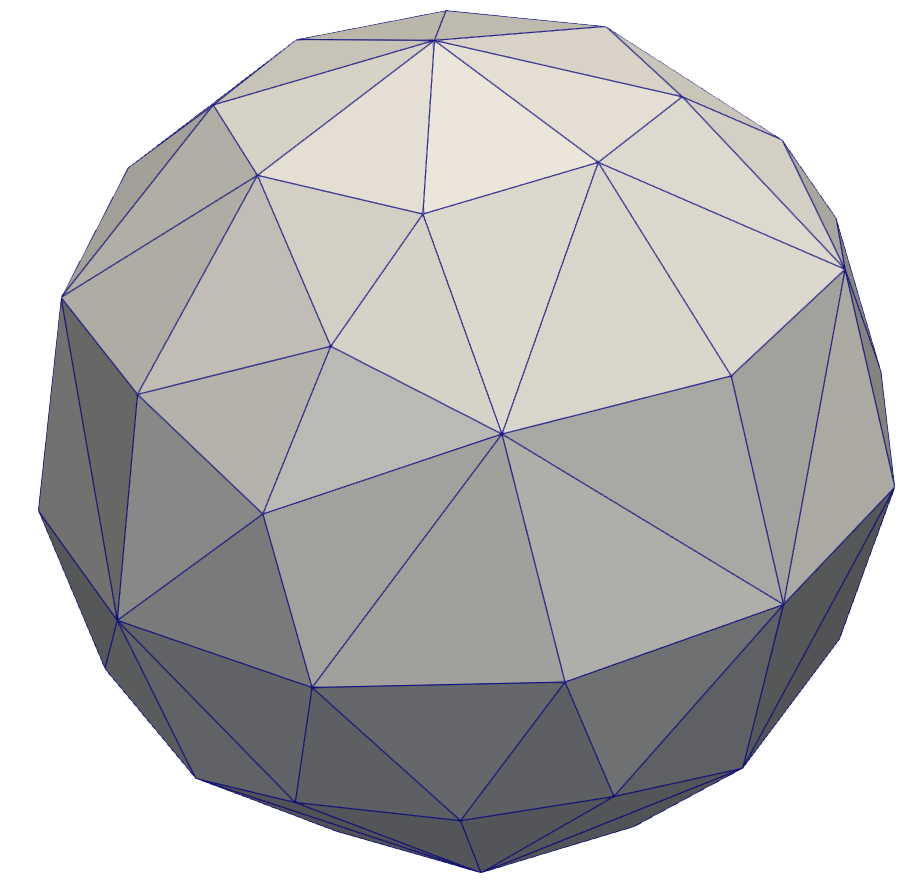}
  \hfill
  \caption{$\mathcal{M}_{h}$}
\end{subfigure}%
\begin{subfigure}{.45\textwidth}

  \centering
 \includegraphics[scale=0.45]{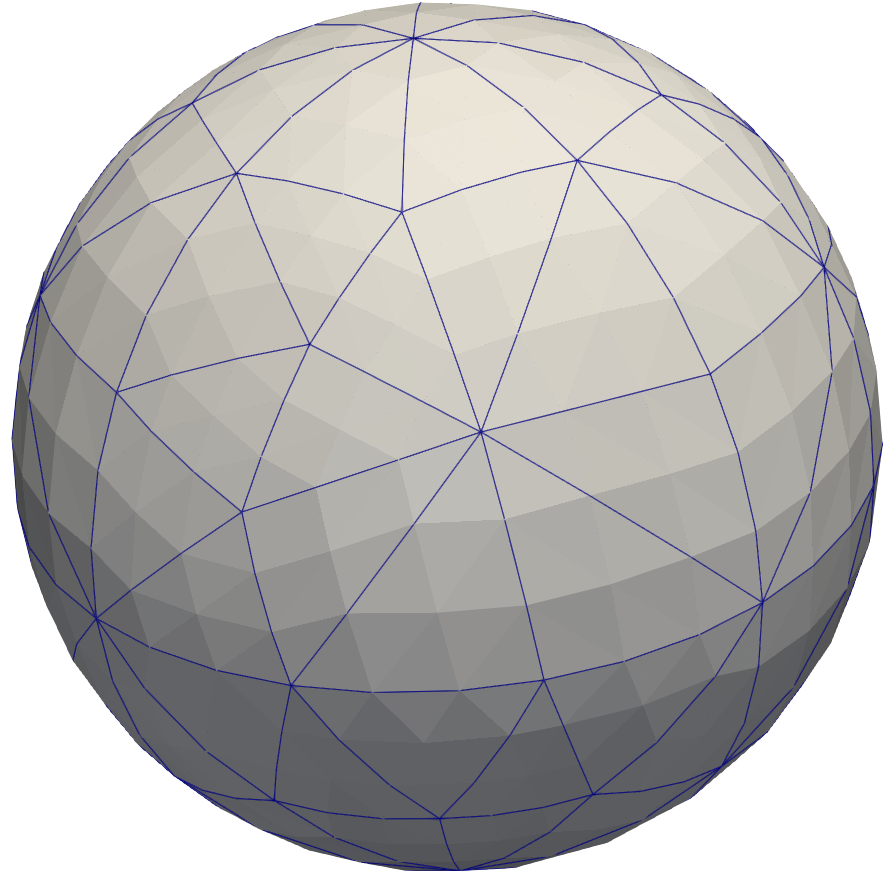}
 \hfill
  \caption{$\mathcal{M}_{h}^{3}$}
\end{subfigure}
\caption{Lagrange parametrization for a torus and a sphere using equidistant nodes on vertices and edges.}
\label{fig:L.param}
\end{figure}
\section{Accuracy of Integration}\label{main}

In this section, we will prove our main theorem about the accuracy of integration when replacing $\mathcal{M}$ by its $k-$th order polygonal approximation $\mathcal{M}_{h}^{k}$ and show that symmetric triangles of the mesh combined with even-degree polynomials improve global errors. This motivates the following definition.
\begin{figure}[h]
\centering
\begin{tikzpicture}
        % Include the image
        \node[inner sep=0pt] at (0,0) {\includegraphics[clip,width=0.7\columnwidth]{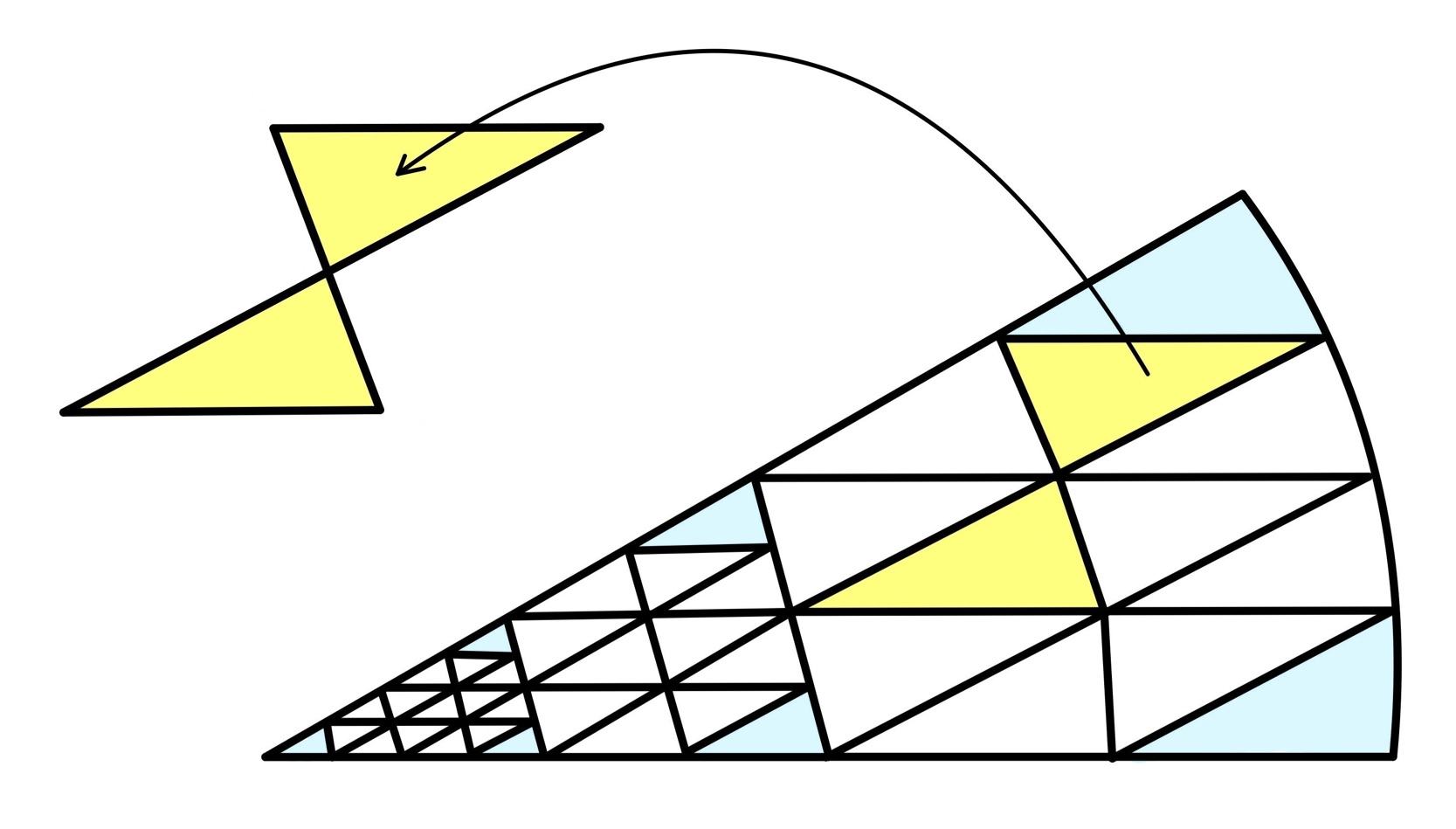}};
        
        % Add text and symbols

        \node[anchor=north west] at (-4.4,2.8) {$\mathbf{x}_{5}$};
        \node[anchor=north west] at (-1.2,2.8) {$\mathbf{x}_{4}$};

        \node[anchor=north west] at (-4.2,1.5) {$\mathbf{x}_{1}$};
           \node[anchor=north west] at (-6.1,0.4) {$\mathbf{x}_{2}$};
            \node[anchor=north west] at (-2.9,0.4) {$\mathbf{x}_{3}$};
             \node[anchor=north west] at (5.3,0.1) {$T\in\mathcal{T}^{1}_{h}$};
    \end{tikzpicture}
\caption{Illustration of bisection refinement and a duo of symmetric triangles obtained from the triangulation $\mathcal{T}^{1}_{h}$.}
\label{fig.symm}
\end{figure}
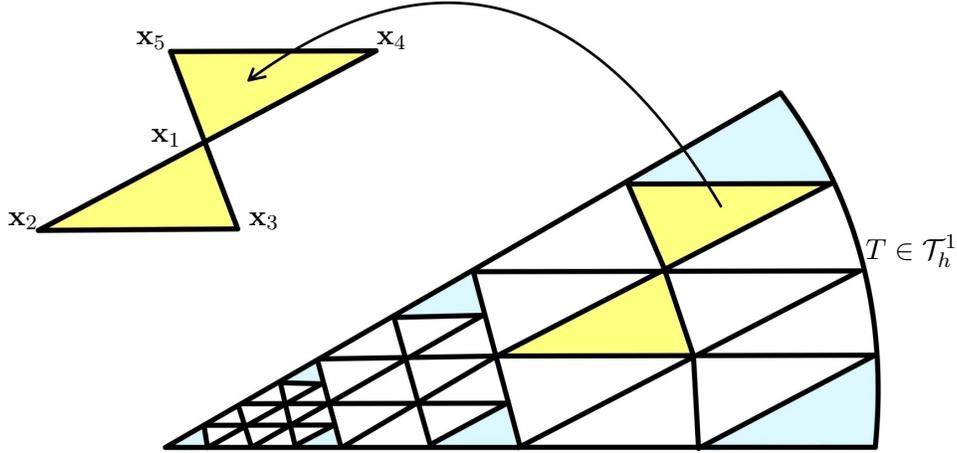

\begin{definition}[symmetric triangles]\label{sym}
Given a collection of triangles $\mathcal{T}^{1}_{h}$ \cite{RUPPERT1995548}, a pair of congruent triangles $T_{\mathbf{x}_{1},\mathbf{x}_{2},\mathbf{x}_{3}}$, $T_{\mathbf{x}_{1},\mathbf{x}_{4},\mathbf{x}_{5}}\in \mathcal{T}^{1}_{h}$ as in Fig.~\eqref{fig.symm} are called symmetric with respect to their common
vertex $\mathbf{x}_{1}$, if they satisfy the following property
\begin{equation}\label{eq.prop}
  \left(\mathbf{x}_{3}-\mathbf{x}_{1}\right)=-\left(\mathbf{x}_{5}-\mathbf{x}_{1}\right),\quad \left(\mathbf{x}_{2}-\mathbf{x}_{1}\right)=-\left(\mathbf{x}_{4}-\mathbf{x}_{1}\right).
\end{equation}
\end{definition}

Traditionally, the specific technique employed for triangulation refinement of a triangulation mesh~$\mathcal{T}^{1}_{h}$ is often regarded as unimportant, provided that the grid size $h\longrightarrow 0$ as $n\longrightarrow \infty$. However, a more intricate picture emerges when the influence of integration process is taken into account. The choice of the \emph{appropriate} type of triangulation can potentially result in a serendipitous cancellation of errors. Now, when examining a conforming triangulation $\mathcal{T}^{1}_{h}$, we refine each triangle $T\in\mathcal{T}^{1}_{h}$ into four smaller triangles using bisection refinement, which involves connecting the midpoints of its edges with straight lines. The new elements are all congruent, and
they are similar to $T$. Repeating the bisection refinement $m-$ times, the total number of triangles at that level will be approximately $\mathcal{O}(2^{2m})$. An advantage of this form of refinement is that each set of mesh points contains those mesh points at the preceding level. During this refinement procedure, we obtain different
triangular elements in size. For example, in Fig.\ref{fig.symm}, you can observe three distinct triangular elements of varying sizes.

Consider $\mathcal{C}_{m},\, m\in \mathbb{N}$ as the set of triangles with equal size within $T\in \mathcal{T}^{1}_{h}$ at the $m$-th refinement level, and let $n_{m}$ represent the number of triangles in this set, then we can write 
\begin{equation*}
\mathcal{C}_{m}=\bigcup_{i=1}^{n_m}\{T^{i}_{h(\frac{1}{2})^{m}}\},\, \text{where}\, h=\diam (T).
\end{equation*}
For example, from the illustration in Fig.~\ref{fig.symm}, we have   
    \begin{equation*}
\mathcal{C}_{0}=T,\;\mathcal{C}_{1}=\bigcup_{i=1}^{n_1}\{T^{i}_{(\frac{h}{2})}\},\;\mathcal{C}_{2} =\bigcup_{i=1}^{n_2}\{T^{i}_{h(\frac{1}{2})^{2}}\} 
,\;\mathcal{C}_{3}= \bigcup_{i=1}^{n_3}\{T^{i}_{h(\frac{1}{2})^{3}}\} ,\;\mathcal{C}_{4} =\bigcup_{i=1}^{n_4}\{T^{i}_{h(\frac{1}{2})^{4}}\}.\,
\end{equation*}
Observe that during each bisection refinement step, at least two triangles fail to meet the Definition \ref{sym}. As an illustration, these triangles can be discerned based on their blue shading in Fig.~\ref{fig.symm}. As a result, the density of symmetric pairs of triangles within $T\in \mathcal{T}^{1}_{h}$ after the $m$-th refinement level is $\mathcal{O}\left(\Tilde{n}_{m}\right)$ while the remaining triangles exhibit a density of $\mathcal{O}\left(\sqrt{\Tilde{n}_{m}}\right)$. Here, $\Tilde{n}_{m}$ denotes the total number of triangles within $T\in \mathcal{T}^{1}_{h}$ after the $m$-th refinement.

\begin{remark}[symmetric triangulation] Any (smoothed closed) surface triangulation mesh that has undergone the refinement process previously described will be referred to as a \emph{symmetric triangulation}. 
\end{remark}

In light of these facts, we state the main theorem of this article, showing that even degree approximation is superior to odd degree approximation for \emph{symmetric triangulations}.  
\begin{theorem}\label{main.thm}
Let $\mathcal{M}$ be a smooth closed  embedded hypersurface and $f\in \mathbf{C}^{k+2}\left(\mathcal{M},\mathbb{R}\right)$. Consider a piecewise linear triangulation $\mathcal{M}_{h}$ with mesh size $h$ of the smooth surface having vertices lie on $\mathcal{M}$ and let $\mathcal{M}^{k}_h$ be the $k-$th order approximation of the smooth surface constructed using local fittings of $\pi\in\mathbf{C}^{k+3}\left(\mathcal{M}_{h},\mathbb{R}^{3}\right)$. Consider a symmetric triangulation, consisting of $\mathcal{O}\left(n\right)=\mathcal{O}\left(h^{-2}\right)$ 
symmetric triangles, whereas the number of non-symmetric triangles is bounded by $\mathcal{O}\left(\sqrt{n}\right)=\mathcal{O}\left(h^{-1}\right)$. Then 
\begin{equation}
\left| \int_{\mathcal{M}} f \, dS - \int_{\mathcal{M}^{k}_{h}} Q_{f,k} \, dS \right|
\leq \begin{cases}
    Ch^{k+2}, & k\equiv 0\ (\textrm{mod}\ 2)\\
    Ch^{k+1}, & k\equiv 1\ (\textrm{mod}\ 2),
  \end{cases} 
\end{equation}
where $Q_{f,k}: \mathcal{M}^{k}_{h}\rightarrow\mathbb{R}$ is a $k-$th order polynomial approximating the integrand $f$.
\end{theorem}
As we approach the proof of Theorem \ref{main.thm}, we need three lemmata.
\begin{lemma}\label{err.tria}
Let $T$ be a triangle in $\mathcal{T}^{1}_{h}$ and  assume that $\psi\in\mathbf{C}^{k+1}\left(T\right)$ and~$k\geq 0$, then we have
\begin{equation}\label{eq.est}
   \norm{\psi-Q_{\psi,k}}_{L^{\infty}\left(T\right)} \leq c h^{k+1} \max^{l,p,q\geq 0}_{l+p+q=k+1}\max_{\left(x,y,z\right)\in T }\Big|\frac{\partial^{k+1}\psi\left(x,y,z\right)}{\partial x^{l}\partial y^{p}\partial z^{q}}\Big|
\end{equation}
with $h=diam(T)$. The constant $c$ depends on $k$, but it is independent of both $\psi$ and $T$.
\end{lemma}
\begin{proof}
In order to prove this lemma, we will use Taylor’s formula in several variables. Let's denote with $\alpha=\left(\alpha_{1},\alpha_{2}\right)$, where $|\alpha|=|\alpha_{1}|+|\alpha_{2}|$, $h:=\left(h_{1},h_{2}\right)=\left(s-0,t-0\right)$ and $\partial_{\alpha_{1}}=\partial_{s},\,\partial_{\alpha_{2}}=\partial_{t}.$\\ As a first step, let's consider the $k-$ order interpolation of the map $\psi$\begin{equation*}    Q_{\psi,k}\left(\hat{p}\right):=\sum_{i=1}^{N\left(2,n\right)} p_{i}\mathcal{L}_{i}^{k}\left(\hat{p}\right).
\end{equation*}
Using Taylor's formula for $\psi$ around the point $\left(0,0\right)$ we obtain:
\begin{equation}\label{eq.lm4}
  \psi\left(\hat{p}\right)=\sum_{|\alpha|=k}\Big(\frac{\partial^{\alpha}}{\alpha!}\psi\left(0\right)h^{\alpha}\Big)+ \frac{1}{k!}\int_{0}^{1}(1-\mu)^{k}\frac{d^{k+1}\psi\left(s\mu,t\mu\right)}{d^{k+1}\mu}d\mu.
\end{equation}
Let us denote $H\left(\hat{p}\right)=    \frac{1}{k!}\int_{0}^{1}(1-\mu)^{k}\frac{d^{k+1}\psi\left(s\mu,t\mu\right)}{d^{k+1}\mu}d\mu$, by interpolating each term in the Eq.~\eqref{eq.lm4} using a polynomial of order $k$ and using the fact that the interpolation of the first term is exact because it is a polynomial of degree~$\leq k$, namely

\begin{equation}
\mathcal{I}_{k}\sum_{|\alpha|=k}\Big(\frac{\partial^{\alpha}}{\alpha!}\psi\left(0\right)h^{\alpha}\Big)=\sum_{|\alpha|=k}\Big(\frac{\partial^{\alpha}}{\alpha!}\psi\left(0\right)h^{\alpha}\Big),
\end{equation}
we obtain
\begin{equation}\label{eq.main1}
Q_{\psi,k}\left(\hat{p}\right)=\sum_{|\alpha|=k}\Big(\frac{\partial^{\alpha}}{\alpha!}\psi\left(0\right)h^{\alpha}\Big)+\sum_{i=1}^{N\left(2,n\right)} H\left(\hat{p}_{i}\right)\mathcal{L}_{i}^{k}\left(\hat{p}\right).
\end{equation}
Subtracting Eq.\eqref{eq.main1} from Eq.\eqref{eq.lm4}, we have
\begin{equation}\label{eq.60}
\psi\left(\hat{p}\right)-Q_{\psi,k}\left(\hat{p}\right)=H\left(\hat{p}\right)-\sum_{i=1}^{N\left(2,n\right)} H\left(\hat{p}_{i}\right)\mathcal{L}_{i}^{k}\left(\hat{p}\right).
\end{equation}
The right hand side of Eq.\eqref{eq.60} can be written
\begin{equation}\label{62}
    H\left(\hat{p}\right)-\sum_{i=1}^{N\left(2,n\right)} H\left(\hat{p}_{i}\right)\mathcal{L}_{i}^{k}\left(\hat{p}\right)=\frac{1}{k!}\int_{0}^{1}(1-\mu)^{k}E\left(\mu;s,t\right)d\mu,
\end{equation}
where $E\left(\mu;s,t\right)$ has the following form
\begin{equation}
    E\left(\mu;s,t\right):=\frac{d^{k+1}\psi\left(s\mu,t\mu\right)}{d^{k+1}\mu}-\sum_{i=1}^{N\left(2,n\right)}\mathcal{L}_{i}^{k}\left(\hat{p}\right) \frac{d^{k+1}\psi\left(s_{i}\mu,t_{i}\mu\right)}{d^{k+1}\mu}.
\end{equation}
It is important to note that the left-hand side of Eq.~\eqref{eq.60} is affected by the behavior of the term $E\left(\mu;s,t\right)$, which itself is influenced by $\frac{d^{k+1}\psi\left(s\mu,t\mu\right)}{d^{k+1}\mu}$. For $k=0$, we have
\begin{equation}
   \max_{0\leq\mu\leq 1,\,\hat{p}\in\sigma}\bigg|\frac{d\psi\left(s\mu,t\mu\right)}{d\mu}\bigg|\leq ch\max \biggl\{
   \norm{\frac{\partial \psi}{\partial x}}_{L^\infty},\norm{\frac{\partial \psi}{\partial y}}_{L^\infty} ,\norm{\frac{\partial \psi}{\partial z}}_{L^\infty}\biggr\},
\end{equation}
where 
\begin{equation*}
    \norm{\frac{\partial \psi}{\partial x}}_{L^\infty}:=\max_{\left(x,y,z\right)\in T} \bigg|\frac{\partial \psi}{\partial x} \bigg|
\end{equation*}
and analogously for $\frac{\partial \psi}{\partial y},\,\frac{\partial \psi}{\partial z}$.
Following the same line of reasoning with the higher-order derivatives of $\psi$, we obtain
\begin{equation}
   \max_{0\leq\mu\leq 1,\,\hat{p}\in\sigma}\bigg|\frac{d^{k+1}\psi\left(s\mu,t\mu\right)}{d^{k+1}\mu}\bigg|\leq ch^{k+1}\max \biggl\{
   \norm{\frac{\partial \psi}{\partial x}}_{L^\infty},\norm{\frac{\partial \psi}{\partial y}}_{L^\infty} ,\norm{\frac{\partial \psi}{\partial z}}_{L^\infty}\biggr\}.
\end{equation}
Using \eqref{62}, yields
\begin{equation}\label{err.1}
        H\left(\hat{p}\right)-\sum_{i=1}^{N\left(2,n\right)} H\left(\hat{p}_{i}\right)\mathcal{L}_{i}^{k}\left(\hat{p}\right)=\mathcal{O}\left(h^{k+1}\right).
\end{equation}
Combining Eq.~\eqref{eq.60} with Eq.~\eqref{err.1}, yields Eq.~\eqref{eq.est}.
\end{proof}
Next, we recall the following Lemma.
\begin{lemma}[\cite{Chien1993}]\label{lem:err}
Let $k$ be an even integer. Let $\psi\left(\hat{p}\right)\in \mathcal{P}_{k+1}$ be a polynomial of degree $k+1$ on $\sigma$, and let $Q_{\psi,k}\left(\hat{p}\right)$ be its interpolant of degree $k$. Then, for each $\hat{p}\in \sigma$
\begin{equation}
    \int_{\sigma}\partial_{s}\left(\psi\left(\hat{p}\right)-Q_{\psi,k}\left(\hat{p}\right)\right)dsdt=0,\quad   \int_{\sigma}\partial_{t}\left(\psi\left(\hat{p}\right)-Q_{\psi,k}\left(\hat{p}\right)\right)dsdt=0.
\end{equation}

\end{lemma}
\begin{lemma}\label{local_error}
Let $\mathcal{M}$ be a smooth closed hypersurface and consider a piecewise linear triangulation $\mathcal{M}_{h}$ with mesh size $h$ of the smooth surface $\mathcal{M}$ and let $\mathcal{M}^{k}_h$ be the $k-$th order approximation of the smooth surface constructed using local fittings of $\pi\in\mathbf{C}^{k+3}\left(\mathcal{M}_{h},\mathbb{R}^{3}\right)$. 
Consider a symmetric triangulation, consisting of $\mathcal{O}\left(n\right)=\mathcal{O}\left(h^{-2}\right)$ 
symmetric triangles, whereas the number of non-symmetric triangles  is bounded by $\mathcal{O}\left(\sqrt{n}\right)=\mathcal{O}\left(h^{-1}\right)$.
Then 
 \begin{equation}\label{local.eq}
\left|\int_{\mathcal{M}}dS-\int_{\mathcal{M}^{k}_{h}}dS\right|
\leq \begin{cases}
    Ch^{k+1}, & k\equiv 1\ (\textrm{mod}\ 2)\\
    Ch^{k+2}, & k\equiv 0\ (\textrm{mod}\ 2).
  \end{cases} 
 \end{equation}
\end{lemma}
\begin{proof}
First, let us write every term in Eq.~\eqref{local.eq} over a reference simplex 
\begin{equation}\label{int.2}
        \int_{\mathcal{M}}dS =\sum_{i=1}^{n}\int_{\sigma}g_{i}dsdt,\quad \int_{\mathcal{M}^{k}_{h}}dS =\sum_{i=1}^{n}\int_{\sigma}\Tilde{g}_{i}dsdt.
\end{equation}
In the same manner, as in  Lemma \ref{err.tria}, we obtain
\begin{equation}\label{eq.main}
    \psi\left(\hat{p}\right)-Q_{\psi,k}\left(\hat{p}\right)=\mathcal{R}_{k+1}+\mathcal{R}_{k+2}+\mathcal{O}\left(h^{k+3}\right),
\end{equation}
where
\begin{equation*}
    \mathcal{R}_{k+1}:=\sum_{|\alpha|=k+1}\Big(\frac{\partial^{\alpha}}{\alpha!}\psi\left(0\right)h^{\alpha}-\sum_{i=1}^{N\left(2,n\right)}h^{\alpha}\frac{\partial^{\alpha}}{\alpha!}\psi\left(0\right)\mathcal{L}_{i}^{k}\left(\hat{p}\right)\Big)
\end{equation*}
\begin{equation*}
   \mathcal{R}_{k+2}:=\sum_{|\alpha|=k+2}\Big(\frac{\partial^{\alpha}}{\alpha!}\psi\left(0\right)h^{\alpha}-\sum_{i=1}^{N\left(2,n\right)}h^{\alpha}\frac{\partial^{\alpha}}{\alpha!}\psi\left(0\right)\mathcal{L}_{i}^{k}\left(\hat{p}\right)\Big).
\end{equation*}
Analogously, an expansion can be given for the errors in the partial derivatives of $\partial_{\alpha_{1}}\psi\left(\hat{p}\right),\, \partial_{\alpha_{2}}\psi\left(\hat{p}\right)$, amounting to apply the derivative on \eqref{eq.main}. Therefore, we have
\begin{equation}
    \partial_{\alpha_{1}}\psi\left(\hat{p}\right)-\partial_{\alpha_{1}}Q_{\psi,k}\left(\hat{p}\right)=\partial_{\alpha_{1}}\mathcal{R}_{k+1}+\partial_{\alpha_{1}}\mathcal{R}_{k+2}+\mathcal{O}\left(h^{k+2}\right)
\end{equation}
\begin{equation}
    \partial_{\alpha_{2}}\psi\left(\hat{p}\right)-\partial_{\alpha_{2}}Q_{\psi,k}\left(\hat{p}\right)=\partial_{\alpha_{2}}\mathcal{R}_{k+1}+\partial_{\alpha_{2}}\mathcal{R}_{k+2}+\mathcal{O}\left(h^{k+2}\right).
\end{equation}
Using Taylor's theorem and expanding about $\left(s=0,t=0\right)$, we obtain
\begin{align}\label{eq.20im}
g_{i}-\Tilde{g}_{i}&:=\sqrt{\det(J^{T}J)} -\sqrt{\det(\Tilde{J}^{T}\Tilde{J})}\\\nonumber
&=\mathcal{E}\left(h^{k+2};\left(\mathbf{x}_{k}-\mathbf{x}_{j}\right),\left(\mathbf{x}_{\ell}-\mathbf{x}_{j}\right)\right)+\mathcal{E}\left(h^{k+3};\left(\mathbf{x}_{k}-\mathbf{x}_{j}\right),\left(\mathbf{x}_{\ell}-\mathbf{x}_{j}\right)\right)+\mathcal{O}\left(h^{k+4}\right),
\end{align}
 where $\mathcal{E}\left(h^{k+2};\left(\mathbf{x}_{k}-\mathbf{x}_{j}\right),\left(\mathbf{x}_{\ell}-\mathbf{x}_{j}\right)\right)$ and $\mathcal{E}\left(h^{k+3};\left(\mathbf{x}_{k}-\mathbf{x}_{j}\right),\left(\mathbf{x}_{\ell}-\mathbf{x}_{j}\right)\right)$, describes the collection of terms with order $k+2, k+3$ in $h$, whose coefficients are combinations of the vertices of triangles with appropriate indices $j,\, \ell$, and $k$. For example for a triangle $T_{\mathbf{x}_{1},\mathbf{x}_{2},\mathbf{x}_{3}}$ the coefficients are combination of $\left(\mathbf{x}_{3}-\mathbf{x}_{1}\right)$ and $\left(\mathbf{x}_{2}-\mathbf{x}_{1}\right)$ both for $\mathcal{E}\left(k+2\right)$ and $\mathcal{E}\left(k+3\right)$. From above computation we see that $\lvert g_{i}-\Tilde{g}_{i}\lvert$ is at least of order $\mathcal{O}\left(h^{k+2}\right) $, confirming also the result obtained in  (\cite{Ray2012}, Theorem 7). However, this result can be further improved assuming that the triangulation mesh comprises symmetric triangles. Thus, integrating Eq.~\eqref{eq.20im} over a reference simplex, we have
  \begin{equation}\label{eq.2}
   \int_{\sigma} \big(g_{i}-\Tilde{g}_{i}\big)dsdt =\int_{\sigma}\mathcal{E}\left(k+2\right)dsdt+\int_{\sigma} \mathcal{E}\left(k+3\right)dsdt+\mathcal{O}\left(h^{k+4}\right).
 \end{equation}

Now if $k$ is odd,~yields
 \begin{equation*}
     \int_{\sigma} \mathcal{E}\left(k+2\right)dsdt\neq 0
 \end{equation*}
then, the left hand side of the Eq.~\eqref{eq.2} is at least of order $\mathcal{O}\left(h^{k+2}\right)$ for every $T\in\mathcal{T}^{1}_{h}$.
Writing Eq.~\eqref{eq.2} with respect to a pair of symmetric triangles $T_{\mathbf{x}_{1},\mathbf{x}_{2},\mathbf{x}_{3}}$ and~$T_{\mathbf{x}_{1},\mathbf{x}_{4},\mathbf{x}_{5}}$, we obtain
 \begin{equation}\label{eq.4}
   \left.\int_{\sigma} \big(g_{i}-\Tilde{g}_{i}\big)dsdt\right|_{T_{\mathbf{x}_{1},\mathbf{x}_{2},\mathbf{x}_{3}}} =\int_{\sigma} \mathcal{E}\left(k+2\right)dsdt+\int_{\sigma} \mathcal{E}\left(k+3\right)dsdt+\mathcal{O}\left(h^{k+4}\right) 
 \end{equation}
  \begin{equation}\label{eq.5}
   \left.\int_{\sigma} \big(g_{i}-\Tilde{g}_{i}\big)dsdt\right|_{T_{\mathbf{x}_{1},\mathbf{x}_{4},\mathbf{x}_{5}}} =\int_{\sigma} \mathcal{E}\left(k+2\right)dsdt+\int_{\sigma} \mathcal{E}\left(k+3\right)dsdt+\mathcal{O}\left(h^{k+4}\right).
 \end{equation}
 From Eq.~\eqref{eq.4} and Eq.~\eqref{eq.5}, we have the following 
 \begin{equation}\label{eq.3}
   \int_{\sigma} \big(g_{i}-\Tilde{g}_{i}\big)dsdt =\int_{\sigma} \mathcal{E}\left(k+3\right)dsdt+\mathcal{O}\left(h^{k+4}\right).
 \end{equation}
 This is due to the fact that the first right hand side integrand of the Eq.~\eqref{eq.4} and Eq.~\eqref{eq.5}, are the collection of terms which are of order $k+2$ in $h$, where the coefficients are the combination of the vertices of triangles, thus each integrand is an odd function. In general for a triangle $T_{\mathbf{x}_{1},\mathbf{x}_{2},\mathbf{x}_{3}}$ we may write:
 \begin{equation}\label{eq.7}
     \mathcal{E}\left(h^{k+2};-\left(\mathbf{x}_{3}-\mathbf{x}_{1}\right),-\left(\mathbf{x}_{2}-\mathbf{x}_{1}\right)\right)=- \mathcal{E}\left(h^{k+2};\left(\mathbf{x}_{3}-\mathbf{x}_{1}\right),\left(\mathbf{x}_{2}-\mathbf{x}_{1}\right)\right).
 \end{equation}
In the same manner for $T_{\mathbf{x}_{1},\mathbf{x}_{4},\mathbf{x}_{5}}$
 \begin{equation}\label{eq.8}
     \mathcal{E}\left(h^{k+2};-\left(\mathbf{x}_{4}-\mathbf{x}_{1}\right),-\left(\mathbf{x}_{5}-\mathbf{x}_{1}\right)\right)=- \mathcal{E}\left(h^{k+2};\left(\mathbf{x}_{4}-\mathbf{x}_{1}\right),\left(\mathbf{x}_{5}-\mathbf{x}_{1}\right)\right).
 \end{equation}
 Using the property \eqref{eq.prop},  Eq.~\eqref{eq.7} and Eq.~\eqref{eq.8}, we obtain
  \begin{equation*}
      \mathcal{E}\left(h^{k+2};\left(\mathbf{x}_{3}-\mathbf{x}_{1}\right),\left(\mathbf{x}_{2}-\mathbf{x}_{1}\right)\right)+\mathcal{E}\left(h^{k+2};\left(\mathbf{x}_{4}-\mathbf{x}_{1}\right),\left(\mathbf{x}_{5}-\mathbf{x}_{1}\right)\right)=0.
 \end{equation*}
The first part of inequality\eqref{local.eq} is obtained by leveraging the fact that the number of symmetric triangles in the mesh $\mathcal{M}_{h}$ is of order $\mathcal{O}\left(n\right)=\mathcal{O}\left(h^{-2}\right)$, while the number of non-symmetric triangles is bounded by 
$\mathcal{O}\left(\sqrt{n}\right)=\mathcal{O}\left(h^{-1}\right)$, in combination with Eq.~\eqref{eq.2} and Eq.~\eqref{eq.3}.

At this point, we notice that if $k$ is even, then
 \begin{equation*}
     \int_{\sigma} \mathcal{E}\left(k+2\right)dsdt=0.
 \end{equation*}
This is due to the Lemma \ref{lem:err} because $\mathcal{E}\left(k+2\right)$ represents errors when integrating a polynomial of order $k+1$. Thus, we have
\begin{equation}\label{eq.new}
   \int_{\sigma} \big(g_{i}-\Tilde{g}_{i}\big)dsdt =\int_{\sigma} \mathcal{E}\left(k+3\right)dsdt+\mathcal{O}\left(h^{k+4}\right).
 \end{equation}
Again, writing Eq.~\eqref{eq.new} with respect to a pair of symmetric triangles, we obtain $T_{\mathbf{x}_{1},\mathbf{x}_{2},\mathbf{x}_{3}}$ and~$T_{\mathbf{x}_{1},\mathbf{x}_{4},\mathbf{x}_{5}}$
 \begin{equation}
   \left.\int_{\sigma} \big(g_{i}-\Tilde{g}_{i}\big)dsdt\right|_{T_{\mathbf{x}_{1},\mathbf{x}_{2},\mathbf{x}_{3}}} =\int_{\sigma} \mathcal{E}\left(k+3\right)dsdt+\mathcal{O}\left(h^{k+4}\right) 
 \end{equation}
  \begin{equation}
   \left.\int_{\sigma} \big(g_{i}-\Tilde{g}_{i}\big)dsdt\right|_{T_{\mathbf{x}_{1},\mathbf{x}_{4},\mathbf{x}_{5}}} =\int_{\sigma} \mathcal{E}\left(k+3\right)dsdt+\mathcal{O}\left(h^{k+4}\right).
 \end{equation}

 Now, if $k$ is even and triangles are pairwise symmetric then the error contributed is 
  \begin{equation}\label{eq.6}
   \int_{\sigma} \big(g_{i}-\Tilde{g}_{i}\big)dsdt =\mathcal{O}\left(h^{k+4}\right).
 \end{equation}
 This is attributed to the fact that for $k$ even, the following holds true.
   \begin{equation*}
    \mathcal{E}\left(h^{k+3};\left(\mathbf{x}_{3}-\mathbf{x}_{1}\right),\left(\mathbf{x}_{2}-\mathbf{x}_{1}\right)\right)+\mathcal{E}\left(h^{k+3};\left(\mathbf{x}_{4}-\mathbf{x}_{1}\right),\left(\mathbf{x}_{5}-\mathbf{x}_{1}\right)\right)=0.
 \end{equation*}
As in the odd case, combining Eq.~\eqref{eq.new} and  Eq.~\eqref{eq.6}, it yields the second part of inequality\eqref{local.eq}.
 \end{proof}

 We can now prove Theorem \ref{main.thm}
 \begin{proof}[Poof of theorem \ref{main.thm}]
As in Lemma \ref{local_error}, expanding the function $f$ using Taylor's formula around the point $\left(0,0\right)$,~yields
\begin{equation}\label{eq_mn}
    f\left(\psi\left(\hat{p}\right)\right)-Q_{f,k}\left(\psi\left(\hat{p}\right)\right)=\mathcal{R}_{k+1}+\mathcal{O}\left(h^{k+2}\right),
\end{equation}
where $\mathcal{R}_{k+1}$ has the form
\begin{equation*}
    \mathcal{R}_{k+1}:=\sum_{|\alpha|=k+1}\Big(\frac{\partial^{\alpha}}{\alpha!}f\left(0\right)h^{\alpha}-\sum_{i=1}^{N\left(2,n\right)}h^{\alpha}\frac{\partial^{\alpha}}{\alpha!}f\left(0\right)\mathcal{L}_{i}^{k}\left(\hat{p}\right)\Big).
\end{equation*}
Integrate both side of the Eq.~\eqref{eq_mn}, we have
\begin{equation}\label{eq_mn1}
    \int_{\sigma}\left(f\left(\psi\left(\hat{p}\right)\right)-Q_{f,k}\left(\psi\left(\hat{p}\right)\right)\right)g_{i}dsdt=\int_{\sigma}\mathcal{R}_{k+1}g_{i}dsdt+\mathcal{O}\left(h^{k+4}\right).
\end{equation}
Using $\lvert g_{i}\lvert=\mathcal{O}\left(h^{2}\right)$, the integral term on the right hand side is at least of order $\mathcal{O}\left(h^{k+3}\right)$.
If $k$ is even due to the presence of symmetric triangles the first term on the right-hand side integral is $0$, so we can compute the left-hand integral with an accuracy of order $\mathcal{O}\left(h^{k+4}\right)$, while for $k$ odd, we have the following
\begin{equation}\label{eq_mn2}
\int_{\sigma}  \left( f\left(\psi\left(\hat{p}\right)\right) - Q_{f,k}\left(\psi\left(\hat{p}\right)\right) \right) g_{i}  \, ds \, dt = \mathcal{O}\left(h^{k+3}\right).
\end{equation}

Assume that $\mathcal{M}_{h}^{k}$ is composed of triangles $\Tilde{T}^{k}$, i.e. $\mathcal{M}_{h}^{k}=\bigcup_{i=1}^{n} \Tilde{T}^{k}_{i} $ and the smooth closed surface, parameterized using 
\eqref{sis1} $\mathcal{M}=\bigcup_{i=1}^{n}V_{i}$. Then 
\begin{equation}\label{int.1}
        \int_{\mathcal{M}}fdS =\sum_{i=1}^{n}\int_{V_i}fdS,\quad \int_{\mathcal{M}^{k}_{h}}Q_{f,k}dS =\sum_{i=1}^{n}\int_{\Tilde{T}_{i}^{k}}Q_{f,k}dS.
\end{equation}
Let us rewrite Eq.~\eqref{int.1} over a reference simplex where the quadrature rules are defined
\begin{equation}\label{int.2}
        \int_{\mathcal{M}}fdS =\sum_{i=1}^{n}\int_{\sigma}f\left(\psi\left(\hat{p}\right)\right)g_{i}dsdt,\quad \int_{\mathcal{M}^{k}_{h}}Q_{f,k}dS =\sum_{i=1}^{n}\int_{\sigma}Q_{f,k}\left(\psi\left(\hat{p}\right)\right) \Tilde{g}_{i}dsdt.
\end{equation}
By making use of Lemma \ref{local_error}, we have
\begin{align*}
\left|\int_{\mathcal{M}}fdS-\int_{\mathcal{M}^{k}_{h}}Q_{f,k}dS\right|&\leq \left|\sum_{i=1}^{n}\int_{\sigma}\Big(f\left(\psi\left(\hat{p}\right)\right)g_{i}-Q_{f,k}\left(\psi\left(\hat{p}\right)\right) \Tilde{g}_{i}\Big)dsdt\right|\\
&\leq \sum_{i=1}^{n}\left|\int_{\sigma}\left(f\left(\psi\left(\hat{p}\right)\right)-Q_{f,k}\left(\psi\left(\hat{p}\right)\right)\right)g_{i}
 +Q_{f,k}\left(\psi\left(\hat{p}\right)\right) \left(g_{i}-\Tilde{g}_{i}\right)dsdt\right|\\
&\leq \sum_{i=1}^{n}\int_{\sigma}\lvert\left(f\left(\psi\left(\hat{p}\right)\right)-Q_{f,k}\left(\psi\left(\hat{p}\right)\right)\right)g_{i}\lvert dsdt\\
&+\int_{\sigma}\lvert Q_{f,k}\left(\psi\left(\hat{p}\right)\right) \left(g_{i}-\Tilde{g}_{i}\right)\lvert dsdt\\
&\leq \left\{\begin{array}{cc}
       Ch^{k+2}, & k\equiv 0\ (\textrm{mod}\ 2)\\
       Ch^{k+1}, & k\equiv 1\ (\textrm{mod}\ 2)
\end{array} \right.
\end{align*}

\end{proof}
For the last inequality, we have used Eq.~\eqref{eq_mn1} and Eq.~\eqref{eq_mn2}.
\section{Results and Discussion}
In order to validate our findings, we developed tests employing the Gauss-Bonnet theorem \cite{pressley2001},\\\cite{spivak1999} for a series of classic smooth closed surfaces
given by the following equations:
\begin{enumerate}
  \item[1)] Ellipsoid \quad $\frac{x^2}{a^2} + \frac{y^2}{b^2} + \frac{z^2}{c^2} = 1$,\quad  $a,b,c \in \R\setminus\{0\}$.
    \item[2)] Torus \quad $(x^2 + y^2 + z^2 + R^2 - r^2)^2 - 4R^2(x^2 + y^2) =0$, \quad $0 <r <R \in \R.$
  \item[3)] Sphere \quad $x^2+y^2+z^2= R^2$,  \quad $R\in \R\setminus\{0\}$.
\end{enumerate} 
The analytic expressions for the Gaussian Curvature are
\begin{enumerate}
  \item[1)] Ellipsoid \quad $K_{\mathrm{Gauss}}= \frac{1}{(abc)^2 \left( \frac{x^2}{a^4} + \frac{y^2}{b^4} + \frac{z^2}{c^4} \right)^2}$,\quad  $a,b,c \in \R\setminus\{0\}$.
 \item[2)] Torus \quad $K_{\mathrm{Gauss}} = \frac{\cos v}{r(R + r \cos v)}$, where we used toric coordinates:
 $$(x,y,z) = \big((R + r \cos \theta) \cos \varphi, (R + r \cos \theta) \sin \varphi, r \sin \theta\big),\;\varphi,\theta \in [0,2\pi).$$
  \item[3)] Sphere \quad $K_{\mathrm{Gauss}} = \frac{1}{R^2}$,  \quad $R\in \R\setminus\{0\}$.
\end{enumerate}
In all experiments we utilize the algorithm of Persson and Strang~\cite{Persson} to
generate Delaunay triangulations\footnote{\cite{geuzaine2009gmsh} presents another viable option for mesh generation.} serving as the initial $\mathcal{M}^{1}_{h}$ surface approximation and a Gaussian quadrature rule of $deg=12$ on each triangle, employing $32$ quadrature points \cite{dunavant1985high}.

% To calculate the experimental order of convergence we use the following formula
% \begin{equation}
%     EOC\left(h_{1},h_{2}\right)=\log\left(\frac{\norm{\int_{\mathcal{M}}fdS-\int_{\mathcal{M}^{k}_{h_{1}}}Q_{f,k}dS}_{L^{\infty}}}{\norm{\int_{\mathcal{M}}fdS-\int_{\mathcal{M}^{k}_{h_{2}}}Q_{f,k}dS}_{L^{\infty}}}\right)\left(\log\frac{h_{1}}{h_{2}}\right)^{-1}.
% \end{equation}
\begin{figure}[ht]
\begin{subfigure}{.47\textwidth}
  \centering
  \includegraphics[clip,width=1\columnwidth]{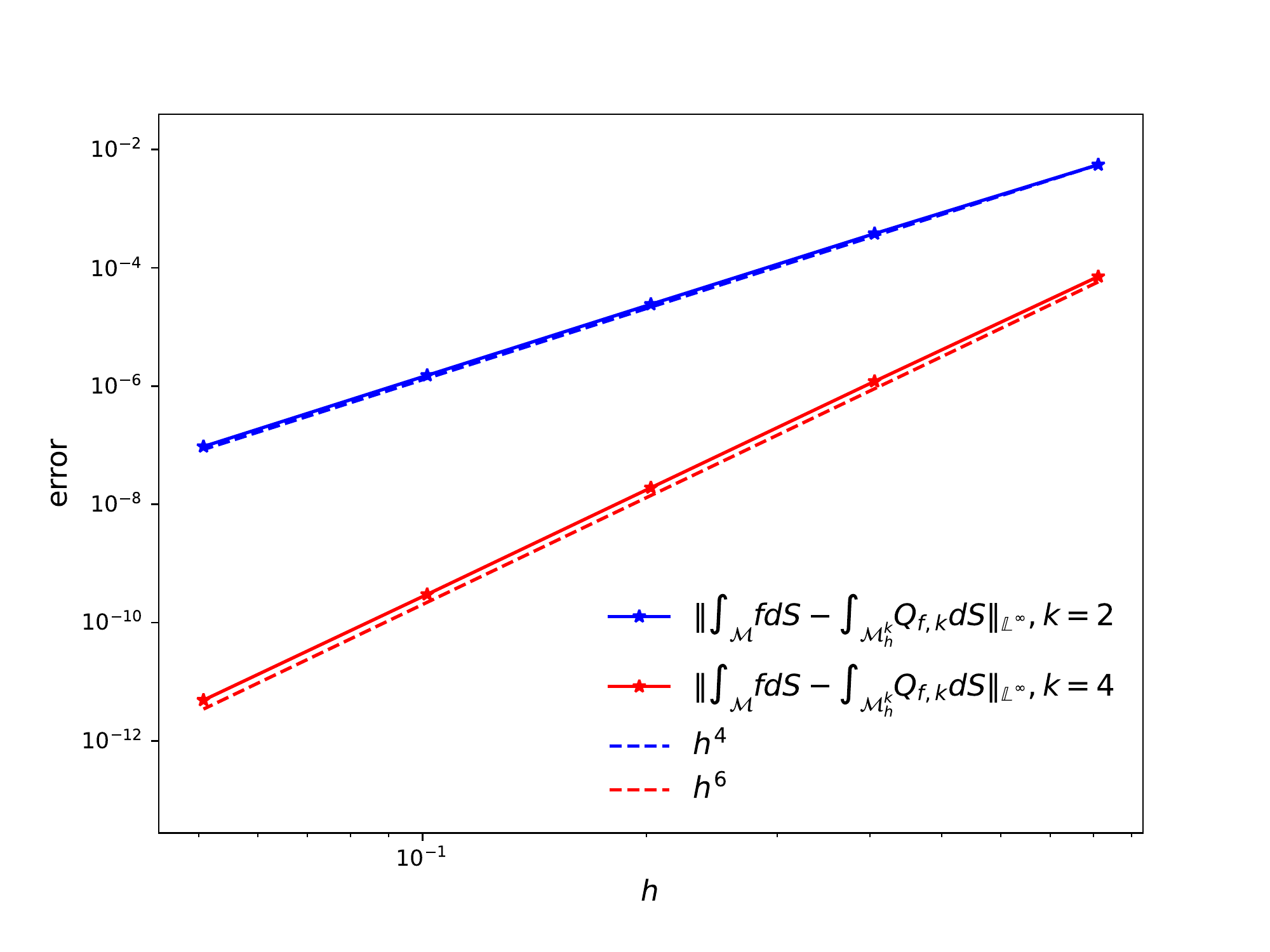}
  \caption{Even order}
\end{subfigure}%
\begin{subfigure}{.47\textwidth}
  \centering
 \includegraphics[clip,width=1\columnwidth]{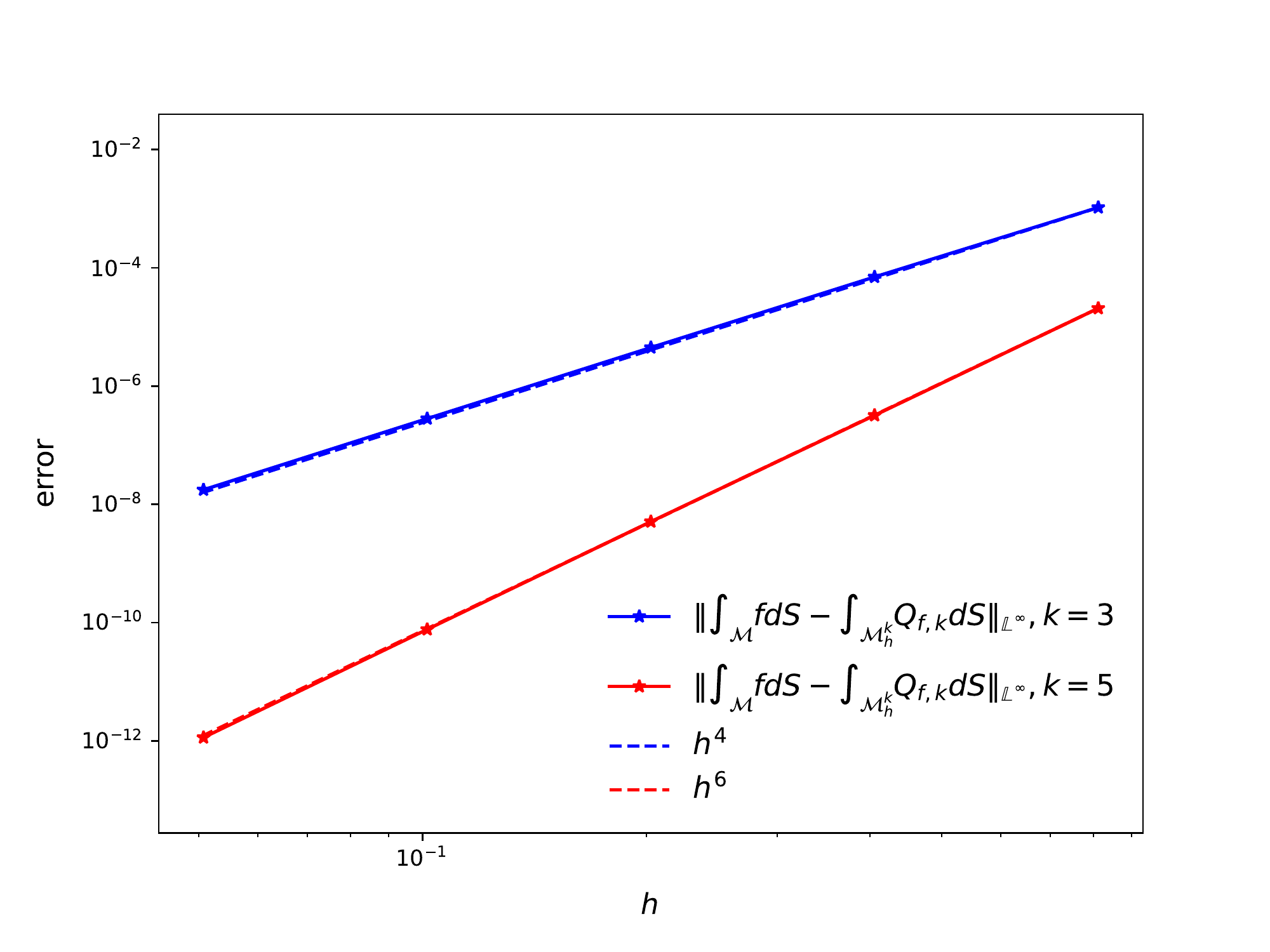}
  \caption{Odd order}
\end{subfigure}
\caption{Relative errors by integrating the Gaussian curvature over the torus with radii $R=2,\,r=1$ with the ideal convergence lines $h^{n}$.}
\label{fig:R}
\end{figure}

\begin{figure}[ht]
\begin{subfigure}{.47\textwidth}
  \centering
  \includegraphics[clip,width=1\columnwidth]{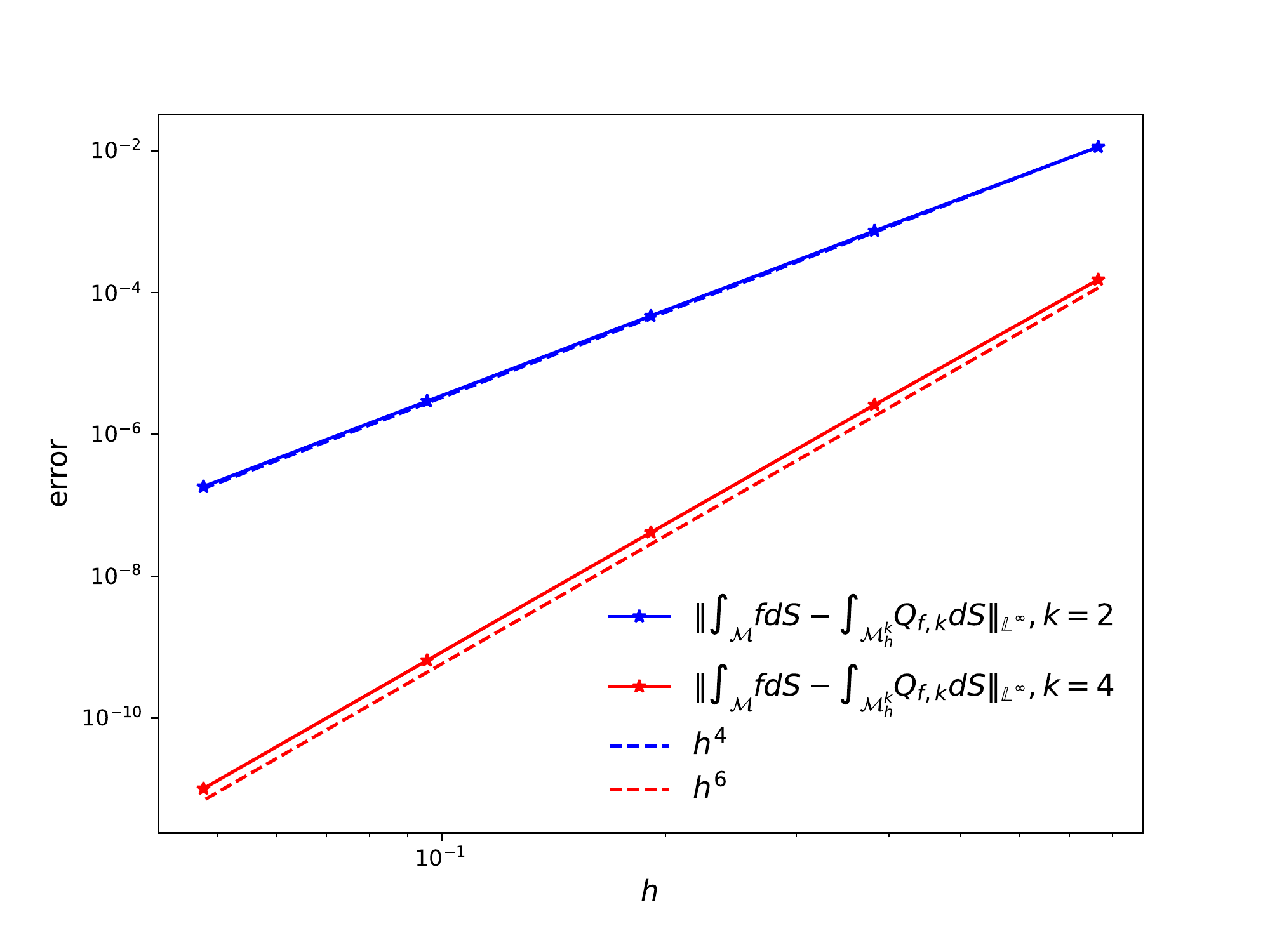}
  \caption{Even order}
\end{subfigure}%
\begin{subfigure}{.47\textwidth}
  \centering
 \includegraphics[clip,width=1\columnwidth]{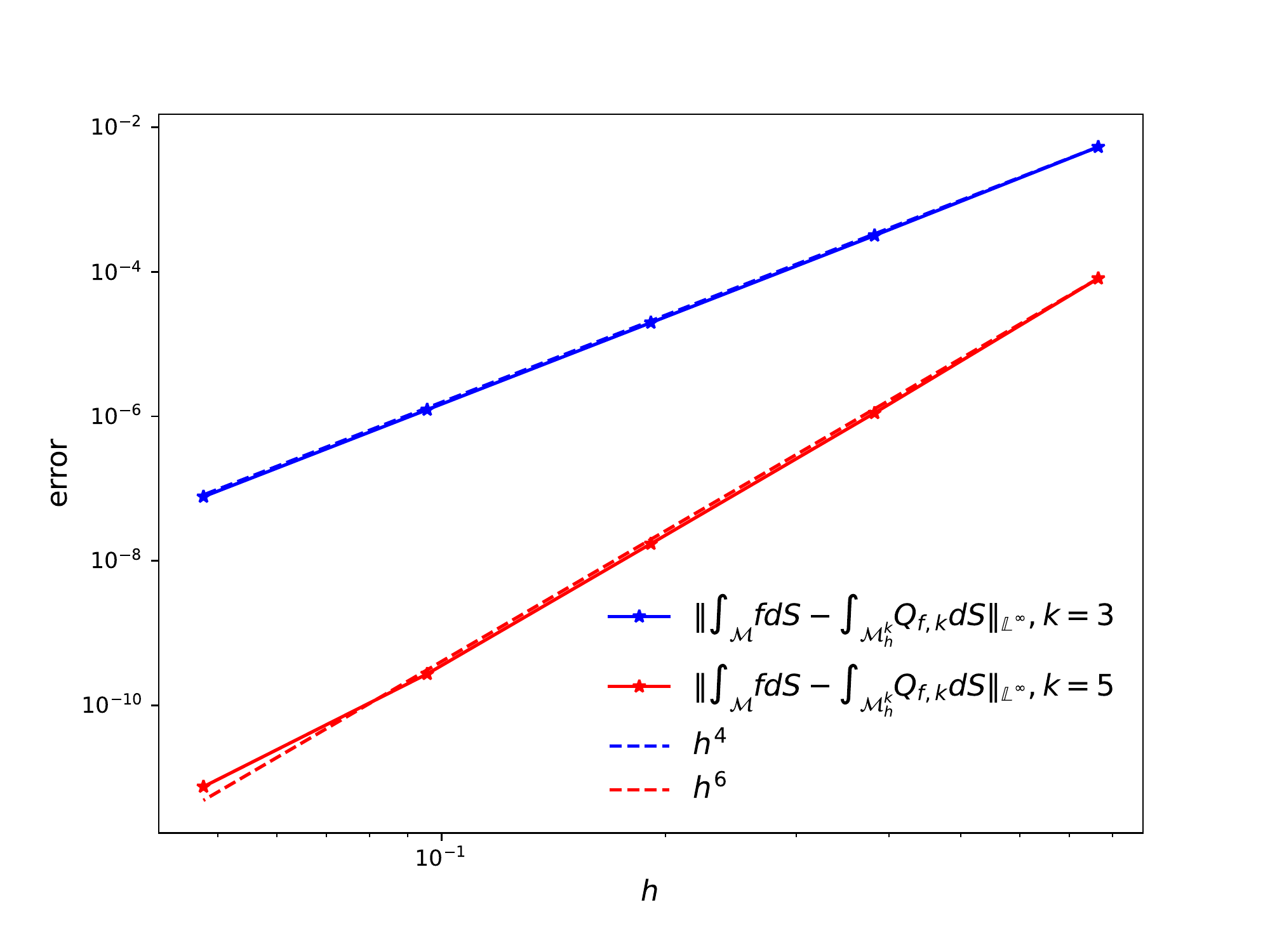}
  \caption{Odd order}
\end{subfigure}
\caption{Relative errors by integrating the Gaussian curvature over the unit sphere with the ideal convergence lines $h^{n}$.}
\label{fig:R1}
\end{figure}
In Fig.~\ref{fig:R} and Fig~\ref{fig:R1} we show the relative errors under mesh refinement. In order to calculate the relative error, we integrate the Gaussian curvature on the manifold and compare it with the predictions of the Gauss-Bonnet Theorem.  The convergence rates shown on the plots coincide, confirming our theoretical results.

However, as can be seen from Fig.~\ref{Fig}, the approach presented in \cite{CurvedGrid} has some limitations, it runs into Runge's phenomenon. This is due to the fact that the Lebesgue constant $\Lambda$ for the set of equidistant points grows exponentially \cite{Mills1992TheLC}.
\begin{figure}[ht]
% \begin{subfigure}{.47\textwidth}
%   \centering
%   \includegraphics[clip,width=1\columnwidth]{torus.png}
%   \caption{Visualization of  Gaussian Curvature for the torus.}
% \end{subfigure}%
% \hfill
% \begin{subfigure}{.47\textwidth}
%   \centering
%  \includegraphics[clip,width=1\columnwidth]{ellipsoid.png}
%    \caption{Visualization of  Gaussian Curvature for the ellipsoid.}
% \end{subfigure}
\begin{subfigure}{.47\textwidth}
  \centering
  \includegraphics[clip,width=1\columnwidth]{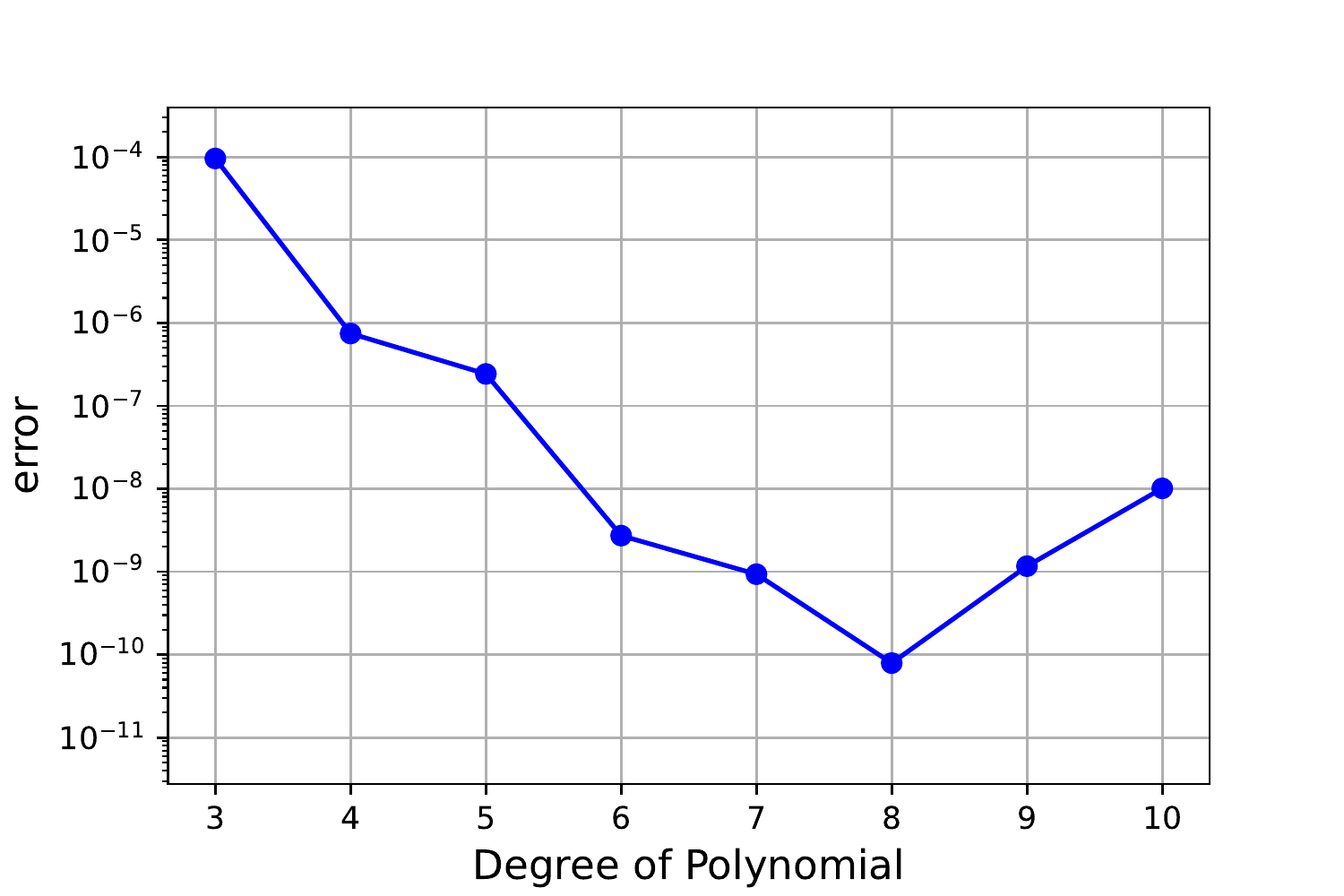}
  \caption{torus with radii $R=2,r=1$}
\end{subfigure}%
\hfill
\begin{subfigure}{.47\textwidth}
  \centering
 \includegraphics[clip,width=1\columnwidth]{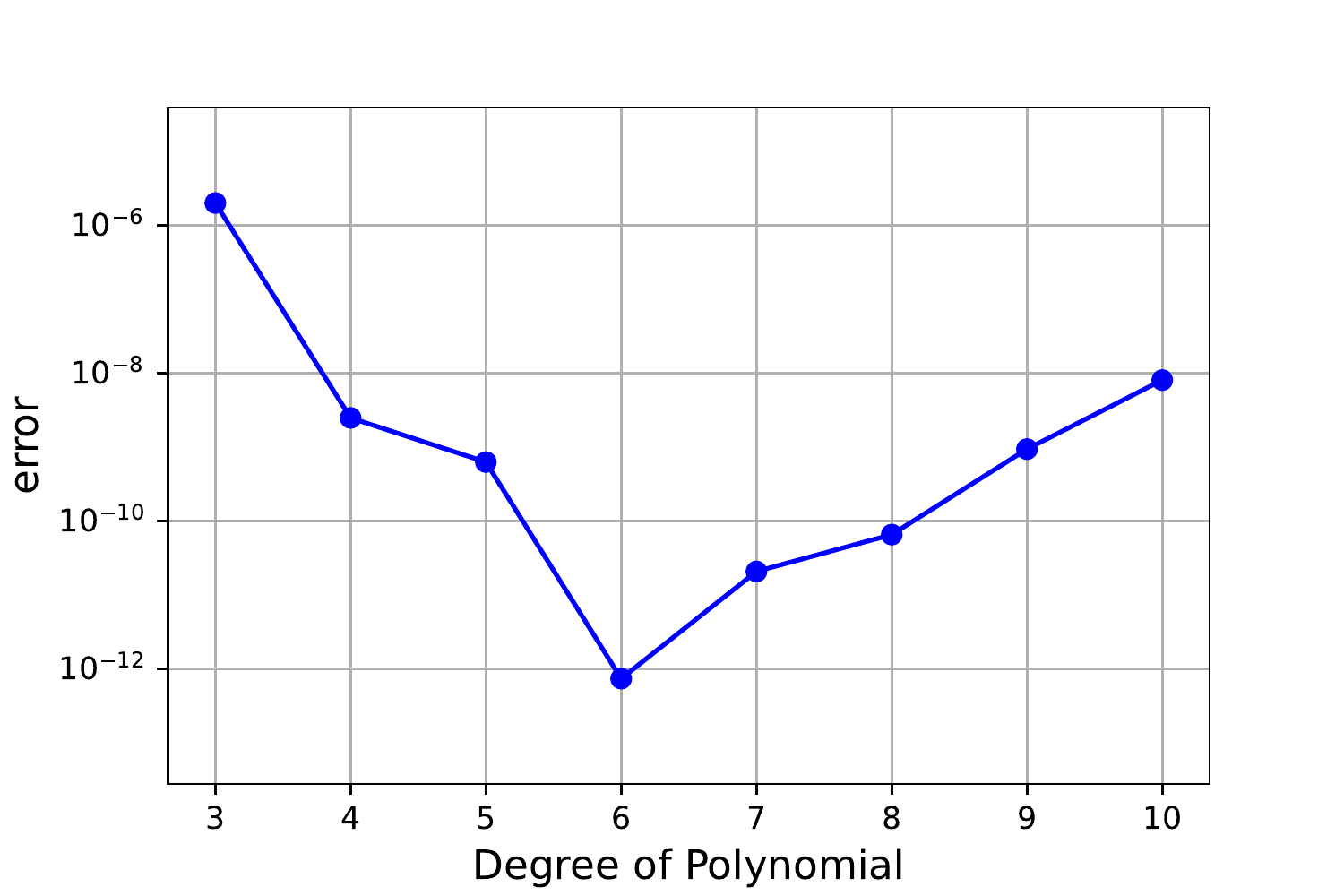}
 \caption{ellipsoid with $a=b=1,c=0.6$.}
\end{subfigure}
\caption{Relative errors by integrating the Gaussian curvature over the torus and the ellipsoid using $N_{\Delta}=2528,\, N_{\Delta}=6152$ respectively.}
\label{Fig}
\end{figure}
To mitigate this effect, we propose an alternative approach. It is well known that under a proper map, the operations (e.g., interpolation, numerical differentiation, and quadrature) on a triangular
element can be performed on the reference square. Below we will state and prove a theorem that we hope will pave the way for developing another powerful numerical integration method for closed surfaces. Prior to presenting the theorem, we first recall the modulus of continuity from \cite{TIMAN1963} (Chapter 3):

Let $\delta_1 \geq 0$ and $\delta_2 \geq 0$ be given. The modulus of continuity for a continuous function $f$ is represented by $\omega\left(f;\delta_{1},\delta_{2}\right)$, which is defined as:
\begin{equation}
    \omega\left(f;\delta_{1},\delta_{2}\right)=\sup_{|x_{1}-x_{2}|\leq \delta_{1},\, x_{1},x_{2}\in [-1,1]} \quad\sup_{|y_{1}-y_{2}|\leq \delta_{2},\, y_{1},y_{2}\in  [-1,1]} |f\left(x_{1},x_{2}\right)-f\left(y_{1},y_{2}\right)|.
\end{equation}
The function $\omega\left(f;\delta_{1},\delta_{2}\right)$ is semi-additive, i.e,
\begin{equation*}
 \omega\left(\delta_{1}+\delta_{2}, \lambda_{1}+\lambda_{2}\right)\leq \omega\left(\delta_{1}+\lambda_{1}\right)+\omega\left(\delta_{2}+\lambda_{2}\right).
\end{equation*}

To address Runge's phenomenon, we propose the use of \emph{Chebyshev--Lobatto nodes} as the interpolation points. Specifically, when working within a one-dimensional interpolation domain of $[-1,1]$, we define our new set of interpolation points as follows
\begin{equation*}
   \Cheb_n = \li\{ \cos\Big(\frac{k\pi}{n}\Big) : 0 \leq k \leq n\re\}\, .
\end{equation*}
These points have a slow increase in the Lebesgue constant as $n \to \infty$, which can be estimated using the following expression
\begin{equation}\label{eq:LEB}
 \Lambda(\Cheb_n)=\frac{2}{\pi}\big(\log(n+1) + \gamma +  \log(8/\pi)\big) + \mathcal{O}(1/n^2)\, ,
\end{equation}
where $\gamma \approx  0.5772$ is the Euler-Mascheroni constant, see \cite{bernstein1931,ehlich,brutman}.
To generate \emph{Chebyshev--Lobatto nodes} on the square  $[-1,1]^{2}$ we consider the tensorial product $\Cheb_n\times\Cheb_n$, yielding  $\Lambda(\Cheb_n\times\Cheb_n) =
 \Lambda(\Cheb_n)^2$ \cite{cohen3}. Now, we can state the following result:

\begin{theorem}\label{important.ftheorem}
For $f\in \mathbf{C}^{k}\left([-1,1]^{2}\right)$, let $Q_{f,n}\left(x,y\right)$ be the interpolant of function $f$ in the \emph{Chebyshev--Lobatto nodes} $\Cheb_n\times\Cheb_n$, we have the estimate
\begin{equation}\label{important.result}
    \norm{f\left(x,y\right)-Q_{f,n}\left(x,y\right)}_{L^{\infty}\left([-1,1]^{2}\right)}\leq \Tilde{C} \ln^{2}\left(n\right)n^{-k} \left(\omega\left(\partial_{x}^{k}f\left(x,y\right);\frac{1}{n},0\right)+\omega\left(\partial_{y}^{k}f\left(x,y\right);0,\frac{1}{n}\right)\right).
\end{equation}
\end{theorem}
\begin{remark}
Based on Theorem \ref{important.ftheorem} in future work, we plan to incorporate recent advances in multivariate interpolation \cite{PIP1,PIP2,MIP,IEEE}, providing a stable approach suppressing Runge's phenomenon by interpolating the map $\psi$ in proper chosen, \emph{transformed Chebyshev-Lobatto nodes}.
\end{remark}
 \begin{proof}[Poof of Theorem \ref{important.ftheorem}]
We consider the interpolation operator $$\mathbf{I}_{[-1,1]^{2}}:\mathbf{C}^{0}\big([-1,1]^{2},\mathbb{R}\big)\rightarrow \mathcal{P}_{2,n} $$
in the Chebyshev--Lobatto nodes $\Cheb_n\times\Cheb_n$, with $\mathcal{P}_{2,n}$ being the space of bivariate polynomials of maximum degree $n$. Its operator norm $\|\mathbf{I}_{[-1,1]^{2}}\| =  \Lambda(\Cheb_n)^2 \in \mathcal{O}(\ln^{2})$ is given by the Lebesgue constant \cite{cohen3}.  
Let us denote with $Q_{f,n}^{*}\left(x,y\right)$ the best polynomial approximation of degree $n.$  The identity theorem for polynomials yields $Q_{f,n}^{*}\left(x,y\right)=\mathbf{I}_{[-1,1]^{2}}\left(Q_{f,n}^{*}\left(x,y\right)\right)$, by making use of Lemma 7.4 in \cite{MIP}, we have
\begin{align}
    &\norm{f\left(x,y\right)-Q_{f,n}\left(x,y\right)}_{L^{\infty}\left([-1,1]^{2}\right)}\nonumber\\
    &\leq \norm{Q_{f,n}^{*}\left(x,y\right)-f\left(x,y\right)}_{L^{\infty}\left([-1,1]^{2}\right)}+ \norm{Q_{f,n}^{*}\left(x,y\right)-Q_{f,n}\left(x,y\right)}_{L^{\infty}\left([-1,1]^{2}\right)}\nonumber\\ 
    &\leq \norm{Q_{f,n}^{*}\left(x,y\right)-f\left(x,y\right)}_{L^{\infty}\left([-1,1]^{2}\right)}+ \norm{\mathbf{I}_{[-1,1]^{2}}\left(Q_{f,n}^{*}\left(x,y\right)-f\left(x,y\right)\right)}_{L^{\infty}\left([-1,1]^{2}\right)}\nonumber\\
    &\leq \left(1+C_{\Lambda} \ln^{2}\left(n\right)\right) \norm{f\left(x,y\right)-Q_{f,n}^{*}\left(x,y\right)}_{L^{\infty}\left([-1,1]^{2}\right)}.\label{best.app2d}
\end{align}
At this point, the multivariate version of Jackson's inequality and semi-additivity of the modulus \cite{TIMAN1963}, gives
\begin{align}
  \norm{f\left(x,y\right)-Q_{f,n}^{*}\left(x,y\right)}_{L^{\infty}\left([-1,1]^{2}\right)}&\leq C \left(2^{k}n^{-k}\omega\left(\partial_{x}^{k}f\left(x,y\right);\frac{2}{n},0\right)+2^{k}n^{-k}\omega\left(\partial_{y}^{k}f\left(x,y\right);0,\frac{2}{n}\right)\right)\nonumber \\
  &\leq C2^{k+1}n^{-k} \left(\omega\left(\partial_{x}^{k}f\left(x,y\right);\frac{1}{n},0\right)+\omega\left(\partial_{y}^{k}f\left(x,y\right);0,\frac{1}{n}\right)\right).\label{jack.ineq}
\end{align}
Combining inequality \eqref{jack.ineq} with \eqref{best.app2d}, we obtain \eqref{important.result}. 
\end{proof}

In light of the
multivariate extension of Jackson’s theorem  \cite{bagby2002multivariate}, we have the following result.
\begin{corollary}\label{corollary}
Let $f\in \mathbf{C}^{k}\left([-1,1]^{2}\right)$,  we have the estimate
\begin{equation}
    \norm{\partial_{x}f\left(x,y\right)-\partial_{x}Q_{f,n}\left(x,y\right)}_{L^{\infty}\left([-1,1]^{2}\right)}\leq  C\left(f;n\right)\ln^{2}\left(n\right)n^{-(k-1)}
\end{equation}
and similarly
\begin{equation}
    \norm{\partial_{y}f\left(x,y\right)-\partial_{y}Q_{f,n}\left(x,y\right)}_{L^{\infty}\left([-1,1]^{2}\right)}\leq C\left(f;n\right)\ln^{2}\left(n\right)n^{-(k-1)},
\end{equation}
where C is a suitable constant (with n), dependent on $f(x,y)$ and $k.$
\end{corollary}
\begin{remark}
If $f\in \mathbf{C}^{k}\left([-1,1]^{2}\right)$ satisfy the Dini–Lipschitz criterion in the sense that
\begin{equation*}
    \lim_{n\longrightarrow\infty}\ln^{2}\left(n\right)\omega\left(f\left(x,y\right);\frac{1}{n},\frac{1}{n}\right)=0\,.
\end{equation*}
Then Eq.~\eqref{important.result} guarantees uniform convergence
\begin{equation*}
   \lim_{n\longrightarrow\infty} \norm{f\left(x,y\right)-Q_{f,n}\left(x,y\right)}_{L^{\infty}\left([-1,1]^{2}\right)}=0.
\end{equation*}
Additionally, it is easy to see that, for $f\in\mathbf{C}^{k}(T)$, we have
\begin{equation}
   \norm{f\left(x,y\right)-Q_{f,n}\left(x,y\right)}_{L^{\infty}\left(T\right)}\leq c\norm{f\left(x,y\right)-Q_{f,n}\left(x,y\right)}_{L^{\infty}\left([-1,1]^{2}\right)}=\mathcal{O}\left(n^{-k}\right).
\end{equation}
\end{remark}
It is anticipated that with this initiative, we can develop an efficient numerical integration method for closed surfaces. This will be addressed in a forthcoming work.

\section*{Declaration of Competing Interest}
The authors declare that they have no known competing financial interests or personal relationships that could have appeared to influence the work reported in this paper.

\section*{Acknowledgement}
We deeply acknowledge Paul Breiding for many inspiring comments and helpful suggestions.

The research of Gentian Zavalani and Michael Hecht was partially funded by the Center of Advanced Systems Understanding (CASUS) which is financed by Germany’s Federal Ministry of Education and Research (BMBF) and by the Saxon Ministry for Science, Culture, and Tourism (SMWK) with tax funds on the basis of the budget approved by the Saxon State Parliament. 

The research of Elima Shehu was funded by the Deutsche Forschungsgemeinschaft (DFG, German Research Foundation), Projektnummer 445466444.

%%%%%%%%%%%%
\bigskip
%%%%%%%%%%%%

\bibliographystyle{alpha}
\bibliography{literatur.bib}

\end{document}